\documentclass[12pt, twoside]{article}
\usepackage{amsmath,amsthm,amssymb}
\usepackage{times}
\usepackage{enumerate}
\usepackage{bm}
\usepackage[all]{xy}
\usepackage{mathrsfs}
\usepackage{amscd}
\usepackage{color}
\def\titlerunning#1{\gdef\titrun{#1}}
\makeatletter
\def\author#1{\gdef\autrun{\def\and{\unskip, }#1}\gdef\@author{#1}}
\def\address#1{{\def\and{\\\hspace*{18pt}}\renewcommand{\thefootnote}{}%
		\footnote {#1}}%
	\markboth{\autrun}{\titrun}}
\makeatother
\def\email#1{e-mail: #1}

\def\keywords#1{\par\medskip
	\noindent\textbf{Keywords.} #1}

\newtheorem{theorem}{Theorem}[section]
\newtheorem{corollary}[theorem]{Corollary}
\newtheorem{lemma}[theorem]{Lemma}
\newtheorem{proposition}[theorem]{Proposition}
\theoremstyle{definition}
\newtheorem{definition}[theorem]{Definition}
\newtheorem{remark}[theorem]{Remark}
\newtheorem{example}[theorem]{Example}
\numberwithin{equation}{section}
\newtheorem{conjecture}{Conjecture}

\xyoption{all}

\def \C {\mathbb{C}}

\def \a {\alpha }
\def \b {\beta}

\def \de {\delta}
\def \De {\Delta}
\def \la {\lambda}
\def \La {\Lambda}
\def\w {\omega}
\def\Om{\Omega}

\def\pa{\partial}
\def\na {\nabla}
\def\Ga{\Gamma}

\begin{document}
	\baselineskip=17pt
	
	\titlerunning{$L^{2}$-Hodge theory and Hopf Conjecture}
	\title{$L^{2}$-Hodge theory on Complete Almost K\"{a}hler Manifolds and the Hopf Conjecture}
	
	\author{Teng Huang, Qiang Tan, Pan Zhang}
	
	\date{}
	
	\maketitle
	
	\address{T. Huang: School of Mathematical Sciences, CAS Key Laboratory of Wu Wen-Tsun Mathematics, University of Science and Technology of China, Hefei, Anhui, 230026, People’s Republic of China; \email{htmath@ustc.edu.cn;htustc@gmail.com}}
	\address{Q. Tan: School of Mathematical Sciences, Jiangsu University, Zhenjiang, Jiangsu 212013, People’s Republic of China; \email{tanqiang@ujs.edu.cn}}
	\address{P. Zhang: School of Mathematical Sciences, Anhui University, Hefei, Anhui, 230026, People’s Republic of China; 
		\email{panzhang20100@ahu.edu.cn}}

	\begin{abstract}
In this article, we develop an $L^{2}$-Hodge theory on complete $2n$-dimensional almost K\"{a}hler manifolds $(X,\w)$. 
In the first part, we establish several identities for various Laplacians, generalized Hodge and Serre dualities, 
a generalized Hard Lefschetz duality, and a Lefschetz decomposition, all restricted to the space $\ker{\De_{\pa}}\cap\ker{\De_{\bar{\pa}}}$ 
of forms of pure bidegree. In the second part, as applications of these identities, we prove vanishing theorems for $L^{2}$-harmonic $(p,q)$-forms on $X$ under some growth assumptions on the K\"{a}her form $\w$. 
We also provide refined $L^{2}$-estimates to sharpen the vanishing theorems in three specific settings. As a final application, the topology of compact almost Kähler manifolds with negative sectional curvature is studied. Under a smallness condition on the Nijenhuis tensor depending on the curvature, the authors prove that the Hirzebruch $\chi_{y}$-genus satisfies $(-1)^{n-p}\chi_{p}(X)\geq1$ for all $p=0,1,\cdots,n$, which in particular implies the Hopf conjecture for the Euler number $(-1)^{n}\chi(X)\geq n+1$. This extends a classical result of Gromov [J. Differential Geom., 1991] from the Kähler to the almost Kähler setting.
	\end{abstract}
	\section{Introduction}

	In complex geometry, Dolbeault cohomology plays a fundamental role in the study of complex manifolds.
	 A classical method for computing it on compact complex manifolds uses the associated spaces of harmonic forms.
	 More precisely, if $X$ is a complex manifold, then the exterior derivative $d$ decomposes as $\pa+\bar{\pa}$, and these operators satisfy $\bar{\pa}^{2}=\pa^{2}=[\pa,\bar{\pa}]=0$. Hence one can define the Dolbeault cohomology and its conjugate as follows:
	$$H^{p,q}_{\bar{\pa}}:=\frac{\ker\bar{\pa}}{{\rm{Im}}\bar{\pa}},\ H^{p,q}_{\pa}:=\frac{\ker\pa}{{\rm{Im}}\pa}.$$
	If $X$ is compact, these spaces are isometric to the kernels of the elliptic operators $\De_{\bar{\pa}}$ and $\De_{\pa}$, respectively.
	
In a non-integrable Hermitian manifold  $X$, where the almost complex structure $J$ is non-integrable, the exterior derivative splits as $\pa+\mu+\bar{\pa}+\bar{\mu}$, and in particular $\bar{\pa}^{2}\neq 0$. Consequently, the standard Dolbeault cohomology and its conjugate are not well-defined. Recently, Cirici and Wilson in \cite{CW21} gave a definition for the Dolbeault cohomology in the non-integrable setting considering also the operator $\bar{\mu}$ together with $\bar{\pa}$. As in the integrable case, one can develop Hodge theory for harmonic forms on $(X,J,\w)$ even in the absence of a cohomological counterpart  (see \cite{CTT,PT21,PT22,TT,TT21}).
	
We define two elliptic selfadjoint differential operators
$$\De_{\bar{\pa}}=\bar{\pa}\bar{\pa}^{\ast}+\bar{\pa}^{\ast}\bar{\pa},\quad \De_{\pa}=\pa\pa^{\ast}+\pa^{\ast}\pa.$$
Cirici and Wilson \cite{CW20} recently proved a generalized Lefschetz decomposition theorem for compact almost K\"{a}hler manifold. 
Denote by $\mathcal{H}^{p,q}_{d}$ the space of harmonic $(p,q)$-forms, i.e. $\ker\De_{d}\cap\Om^{p,q}$. In \cite[Corollary 5.4]{CW20}), they showed that for a compact almost K\"{a}hler manifold $(X,J,\w)$, 
$$\mathcal{H}^{p,q}_{d}=\bigoplus_{r\geq\max\{p+q-n,0\}}L^{r}(\mathcal{H}^{p-r,q-r}_{d}\cap P^{p-r,q-r}).$$
	
In this article, we consider the spaces of ($\bar{\pa}$, $\pa$)-harmonic forms given by the intersections
$$\mathcal{H}^{p,q}_{(2);\bar{\pa}}\cap\mathcal{H}^{p,q}_{(2);\pa}.$$
These are identified with the kernel of the self-adjoint elliptic operator $\De_{\pa}+\De_{\bar{\pa}}$ acting on $\Om^{p,q}_{(2)}(X)$. We denote by
	$$l^{p,q}_{(2)}:= \dim (\mathcal{H}^{p,q}_{(2);\bar{\pa}}\cap\mathcal{H}^{p,q}_{(2);\pa})=\dim\ker(\De_{\bar{\pa}}+\De_{\pa})\cap\Om^{p,q}_{(2)}$$
	the dimensions of these spaces. In the integrable case, since $\De_{\pa}=\De_{\bar{\pa}}=\frac{1}{2}\De_{d}$, these numbers correspond to the Hodge numbers of a compact Kähler manifold. Since $\De_{\pa}+\De_{\bar{\pa}}:\Om^{p,q}\rightarrow\Om^{p,q}$, there is an orthogonal direct sum decomposition
	$$\ker(\De_{\pa}+\De_{\bar{\pa}})\cap\Om^{k}_{(2)}=\bigoplus_{p+q=k}\ker(\De_{\pa}+\De_{\bar{\pa}})\cap\Om^{p,q}_{(2)}.$$	
In the integrable case, the spaces of harmonic forms give rise to the well-known Hodge decomposition. For almost Kähler manifolds, however, the lack of integrability introduces additional complexity. Nevertheless, the intersection of the kernels of $\De_{\bar{\pa}}$ and $\De_{\pa}$ still exhibits rich structure, as shown below.	
	\begin{theorem}\label{T4}
	 For any complete almost K\"{a}hler manifold of dimension $2n$, and for all $(p,q)$, the following dualities hold:\\
	\noindent\textbf{(1) (Complex conjugation).} We have equalities
	$$\ker{\De_{\pa}}\cap\ker{\De_{\bar{\pa}}}\cap\Om^{p,q}_{(2)}(X)=\ker{\De_{\pa}}\cap\ker{\De_{\bar{\pa}}}\cap\Om^{q,p}_{(2)}(X).$$
	\noindent\textbf{(2) (Hodge duality).} The Hodge $\star$-operator induces isomorphisms
	$$\star:\ker{\De_{\pa}}\cap\ker{\De_{\bar{\pa}}}\cap\Om^{p,q}_{(2)}(X)\rightarrow\ker{\De_{\pa}}\cap\ker{\De_{\bar{\pa}}}\cap\Om^{n-q,n-p}_{(2)}(X).$$
	\noindent\textbf{(3) (Serre duality).} There are isomorphisms:
	$$\ker{\De_{\pa}}\cap\ker{\De_{\bar{\pa}}}\cap\Om^{p,q}_{(2)}(X)\cong\ker{\De_{\pa}}\cap\ker{\De_{\bar{\pa}}}\cap\Om^{n-q,n-p}_{(2)}(X).$$
	\noindent\textbf{(4) (Generalized Hard Lefschetz Duality).} The operators $\{L,\La,H=[L,\La]\}$ define a finite dimensional representation of $\mathfrak{sl}(2,\C)$ on
		$$\bigoplus_{p,q\geq0}\ker{\De_{\pa}}\cap\ker{\De_{\bar{\pa}}}\cap\Om^{p,q}_{(2)}(X).$$
		Moreover, for all $0\leq p\leq k\leq n$,
		$$L^{n-k}:\ker{\De_{\pa}}\cap\ker{\De_{\bar{\pa}}}\cap\Om^{p,k-p}_{(2)}(X)\stackrel{\cong}{\longrightarrow}\ker{\De_{\pa}}\cap\ker{\De_{\bar{\pa}}}\cap\Om^{p+n-k,n-p}_{(2)}(X)$$
		are isomorphisms. Furthermore, for any $p,q$ we have an orthogonal direct sum decomposition
		$$\ker{\De_{\pa}}\cap\ker{\De_{\bar{\pa}}}\cap\Om^{p,q}_{(2)}(X)=\bigoplus_{j\geq\max\{p+q-n,0\}}L^{j}(\ker{\De_{\pa}}\cap\ker{\De_{\bar{\pa}}}\cap P^{p-j,q-j}_{(2)}(X)),$$
		where $$P^{r,s}_{(2)}(X)=\ker\La\cap\Om^{r,s}_{(2)}(X).$$
	\end{theorem}

	\begin{remark}
		For any compact almost K\"{a}hler manifold $X$, Cirici and Wilson \cite{CW20} introduced the  $\de$-Laplacian
		$$\De_{\de}=\de\de^{\ast}+\de^{\ast}\de,$$
		and defined the space of $\de$-harmonic forms as
		\begin{equation}\label{E28}
		\mathcal{H}^{p,q}_{\de}=\ker\De_{\de}\cap\Om^{p,q}=\ker\de\cap\ker\de^{\ast}\cap\Om^{p,q}.
		\end{equation}
	 However, in the non-compact setting, for any $\a\in\Om^{p,q}_{(2)}$, $\de\a$ and $\de^{\ast}\a$ may not belong to $L^{2}$ when $\de=\mu,\bar{\mu}$. Hence, the definition of   $\mathcal{H}^{p,q}_{(2);\de}$ is not valid in general, and the equality in (\ref{E28}) may fail.
	\end{remark}
	Let $(X,J,\w)$ be a complete $2n$-dimensional almost K\"{a}hler manifold. A fundamental question, relevant to both topology and function theory, is:  when are there non-trivial harmonic forms on $X$, in the various bidegree $(p,q)$ determined by the almost complex structure? When $X$ is not compact, we denote by $\Om^{p,q}_{(2)}(X)$ the $L^{2}$-forms of type $(p,q)$ on $X$ and $\mathcal{H}^{p,q}_{(2);d}(X)$ the harmonic forms in $\Om^{p,q}_{(2)}(X)$. One version of this basic question is: what is the structure of $\mathcal{H}^{p,q}_{(2);d}(X)$?
	
	The Hodge theorem for compact manifolds states that every de Rahm cohomology class of a compact manifold $X$ is represented by a unique harmonic form. That is, the space of solutions to the differential equation $\De_{d}\a=0$ on smooth forms over $X$ is a space that depends on the metric on $X$. This space is canonically isomorphic to the purely topological de Rahm cohomology space of $X$. The study of $\mathcal{H}^{p,q}_{(2);d}(X)$, a question of so-called $L^{2}$-cohomology of $X$, is rooted in the attempt to extend Hodge theory to non-compact manifolds. The study of the $L^{2}$-harmonic forms on a complete Riemannian manifold is a very  fascinating and important subject.  There are numerous partial results about the $L^{2}$-cohomology of non-compact manifold (see \cite{And,Car,Dod}), but this extension is not yet complete.  When $J$ is integrable in $X$, then $X$ is K\"{a}hlerian. The study of $\mathcal{H}^{p,q}_{(2);d}(X)$ is one of the focal points in complex geometry \cite{CF,Don94,Don97,DF,Gro,McN02,McN03} and the references therein provide a good view on the subject.
	
	The second purpose of this paper is to prove some vanishing results on $\mathcal{H}^{p,q}_{(2);d}(X)$ when $p+q\neq n$, 
	under a growth assumption on a primitive of $\w$ does not grow too fast at infinity. 
	\begin{theorem}\label{T2}
		Let $(X,J,\w)$ be a complete $2n$-dimensional almost K\"{a}hler manifold which is $(L,c)$-vanishingly exhaustible. There exists a uniform positive constant $C(n)$  depends only on $n$ such that if $c\leq C(n)$, then for any $p+q\neq n$,
		$$\ker{\De_{\bar{\pa}}}\cap\ker{\De_{\pa}}\cap\Om^{p,q}_{(2)}(X)=0.$$
		In particular,
		$$\mathcal{H}^{p,q}_{(2);d}(X)=0.$$
	\end{theorem}

	Under the assumption that the Kähler form $\w$ satisfies a certain global geometric condition, the above conclusion provides a vanishing theorem for $L^{2}$-harmonic forms. To further illustrate the sharpness of our results, we now present sharper $L^{2}$-estimates that remain valid under more concrete assumptions, such as $\w$ being $d$(bounded).

	\begin{theorem}[=Theorem \ref{T5}+Corollary \ref{C2}+Corollary \ref{C3}]\label{T11}
Let $(X,J,\omega)$ be a complete $2n$-dimensional almost K\"ahler manifold. 
Suppose that there exists a bounded $1$-form $\theta$ and a constant $c\ge 0$ such that
\[
\sup_{x\in X}\bigl|\omega-(d\theta)^{1,1}\bigr|(x)\le c.
\]
Then for any $k\neq n$ and any $\alpha\in\Omega_0^k(X)$, the following estimates hold:

\noindent\textbf{(1) (Estimate for $\Delta_\partial+\Delta_{\bar\partial}$).}
We have
\[
\|\alpha\|_{L^2}\bigl(1-c(n,k)c\bigr)\le c(n,k)\|\theta\|_{L^\infty(X)}\bigl((\Delta_\partial+\Delta_{\bar\partial})\alpha,\alpha\bigr)^{\frac12},
\]
where $c(n,k)>0$ is a constant depending only on $n$ and $k$. 
If $c(n,k)c<1$, then
\[
\ker(\Delta_\partial+\Delta_{\bar\partial})\cap\Omega_{(2)}^k(X)=\{0\}.
\]

\noindent\textbf{(2) (Estimate for $\Delta_d$).}
We have
\[
\|\alpha\|_{L^2}^2\bigl(1-c(n,k)c\bigr)\bigl(1-c(n,k)c-4c(n,k)\|\theta\|_{L^\infty(X)}\|N_J\|_{L^\infty(X)}\bigr)
\le c(n,k)^2\|\theta\|_{L^\infty(X)}^2(\Delta_d\alpha,\alpha).
\]
If $c(n,k)\bigl(c+4\|\theta\|_{L^\infty(X)}\|N_J\|_{L^\infty(X)}\bigr)<1$, then
\[
\mathcal{H}_{(2);d}^k(X)=\{0\}.
\]

\noindent\textbf{(3) (Estimate for $\Delta_\partial$ or $\Delta_{\bar\partial}$).}
For $\bullet=\partial$ or $\bar\partial$,
\[
\|\alpha\|_{L^2}^2\Bigl(\bigl(1-c(n,k)c\bigr)^2-2c(n,k)^2\|\theta\|_{L^\infty(X)}^2\|N_J\|_{L^\infty(X)}^2\Bigr)
\le 2c(n,k)^2\|\theta\|_{L^\infty(X)}^2(\Delta_\bullet\alpha,\alpha).
\]
If $c(n,k)\bigl(c+\sqrt2\,\|\theta\|_{L^\infty(X)}\|N_J\|_{L^\infty(X)}\bigr)<1$, then
\[
\mathcal{H}_{(2);\bullet}^k(X)=\{0\}.
\]
\end{theorem}

In the final part of the article, we study the topology and geometry of the compact almost K\"{a}hler manifold with negative sectional curvature. We begin by recalling a well-known conjecture related to the negativity of Riemannian sectional curvature.
	\begin{conjecture}[Hopf Conjecture]
		Let $X$ be a closed $2n$-dimensional Riemannian manifold with sectional curvature $\sec$. Then
		\begin{equation*}
		\left\{
		\begin{aligned}
		(-1)^{n}\chi(X)>0, &\ \mathrm{if}\ \sec<0,\\
		(-1)^{n}\chi(X)\geq0,&\ \mathrm{if}\ \sec\leq0.\\
		\end{aligned}  
		\right.
		\end{equation*}	
	\end{conjecture}
	This is true for $n=1$ and $2$ as the Gauss–Bonnet integrands in these two low dimensional cases have the desired sign.
	 However, in higher dimensions, it is known that the sign of the sectional curvature does not determine the sign of the Gauss--Bonnet--Chern integrand.  For some compact Riemannian manifolds X of dimension $2n$ with some suitable pinched negative sectional curvature, the Euler number of these manifolds had been studied by many authors (cf. \cite{DX,JX}).

Let $(X, g)$ be a closed Riemannian manifold and $\pi:(\tilde{X},\tilde{g})\rightarrow(X,g)$ be its universal covering with $\tilde{g}=\pi^{\ast}g$. A differential form $\a$ on $X$ is called $\tilde{d}$(bounded) if its lift $\pi^{\ast}\a$ is a $d$(bounded) form on $(\tilde{X},\tilde{g})$. Gromov observed that if $(X,g)$ is a complete simply-connected manifold with strictly negative sectional curvature, then every smooth bounded closed form of degree $k\geq2$ is $d$(bounded). He subsequently proved the Hopf conjecture in the Kähler case by exploiting  Kähler identities. A fundamental obstruction to extending Gromov's argument beyond the integrable setting lies in the failure of the standard Kähler identities. In the non-integrable case, the crucial commutation relations that underpin Gromov's $L^{2}$-estimates no longer hold. Our approach overcomes these difficulties by establishing a new family of $L^{2}$-estimates that quantitatively absorb the non-integrability of the almost complex structure.
	
	In the symplectic case, inspired by Kähler geometry, one can analogously define symplectic hyperbolic manifolds. A closed almost K\"{a}hler manifold $(X,\w)$ is called symplectic hyperbolic if the lift $\tilde{\w}$ of $\w$ to the universal covering $(\tilde{X},\tilde{\w})\rightarrow(X,\w)$ is $d$(bounded)  on $(\tilde{X},\tilde{\w})$ \cite{Hua22,Hua23}. Hind and Tomassini \cite{HT} constructed a $d$(bounded) complete almost K\"{a}hler manifold $X$ satisfying $\mathcal{H}^{1}_{(2)}(X)\neq\{0\}$  by employing methods from  contact geometry. 
	
Denote by $h^{k}_{(2)}(X)$ the $k$-th $L^{2}$-Betti number of Riemannian manifold $X$. In addition to the Hopf conjecture, another long-standing problem concerning negatively curved manifolds is the Singer conjecture (cf.\cite[Conjecture 2]{Dod}), which predicts the vanishing of all middle-degree $L^{2}$-Betti numbers (cf.\cite{DX,JX}).  
\begin{conjecture}[Singer Conjecture]
	Let $X$ be a closed $2n$-dimensional Riemannian manifold with negative sectional curvature. Then
	\begin{equation*}
	\left\{
	\begin{aligned}
	h^{k}_{(2)}(X)=0, &\quad k\neq n,\\
	h^{n}_{(2)}(X)>0.& \\
	\end{aligned}  
	\right.
	\end{equation*}	
\end{conjecture}
It is well known that the Singer conjecture implies the Hopf conjecture via the Euler–Poincaré formula
$$\chi(X)=\sum_{k=0}^{2n}(-1)^{k}h_{(2)}^{k}(X).$$
The first author of this article proved the Singer conjecture under a certain condition on the Nijenhuis tensor in \cite{Hua23}.
	
The main application in our article is that we can confirm that the Hopf conjecture is correct in the case of almost K\"{a}hler manifold $X$ with small Nijenhuis tensor. A special case is that Nijenhuis tensor vanishes, i.e., the manifold $X$ is K\"{a}herian (see \cite{Gro}). The primary application of our analytic estimates is to extend Gromov's result in \cite{Gro} to the almost K\"ahler setting. We establish that the Hopf conjecture remains valid for almost K\"ahler manifolds provided the Nijenhuis tensor is bounded by a constant depending on the sectional curvature. In particular, we prove that the components of the Hirzebruch $\chi_{y}$-genus satisfy the inequality $(-1)^{n-p}\chi_{p}(X) \geq 1$ for all $p = 0, 1, \dots, n$. This not only confirms the Hopf conjecture for the Euler number but also yields stronger topological constraints on all $\chi_{p}$. Our approach combines the $L^{2}$-estimates established in Theorem \ref{Thm6} with a refined vanishing theorem for the operator $\bar{\partial}+\bar{\partial}^{\ast}$ and Atiyah's $L^{2}$-index theorem \cite{Ati}.
\begin{theorem}\label{Thm2}
	Let $(X,J,\w)$ be a closed $2n$-dimensional almost K\"{a}hler manifold with  negative sectional curvature, 
	i.e. there exists a constant $K>0$ such that
	$$\mathrm{sec}\leq -K.$$
	If the Nijenhuis tensor of $X$ satisfies
	\[|N_J|^{2}\leq C(n)K,\]
	where $C(n)$ is a uniform positive constant, then 	
	\begin{equation*}
	(-1)^{n-p}\chi_{p}(X)\geq1,\ \forall p=0,\cdots,n.
	\end{equation*}
	In particular, the Euler number of $X$ obeys
	\[(-1)^{n}\chi(X)\geq n+1.\]
\end{theorem}
This theorem is one of the main achievements of the article. It shows that even in the non-integrable setting, the Hopf conjecture remains valid provided the Nijenhuis tensor is sufficiently small relative to the curvature. The constant $C(n)$ is explicit and depends only on the complex dimension.
\begin{remark}
	(1) According to the identity for the Nijenhuis tensor  provided  in Proposition \ref{Pro4},  the hypothesis of Theorem \ref{Thm2} is equivalent to the curvature condition
	\[\hat{W}(\w^{\sharp},\w^{\sharp})-\frac{n-1}{2n-1}s_{g}\leq\frac{C(n)}{16}K.\]
	(2) If for all $i,j$, the complex sectional curvature satisfies
	\[
	\bigl(\mathfrak{R}^{\mathbb{C}}(Z_{i}\wedge Z_{j}),\overline{Z_{i}\wedge Z_{j}}\bigr) \geq -\frac{C(n)}{32(n^{2}-n)}K,
	\]
	then \[|N_{J}|^{2}\leq C(n)K .\] Note that any smooth Kähler metric satisfies  $\mathfrak{R}^{\C}(Z_{i}\wedge Z_{j}),\overline{Z_{i}\wedge Z_{j} })=0$ for all $i,j$ (cf.\cite{Her00}).
\end{remark}
\noindent\textbf{Organization of the paper.} The article is organized as follows. Section 2 recalls basic facts about almost K\"ahler manifolds and the decomposition of the exterior derivative. In Section 3, we develop the $L^{2}$-Hodge theory on complete almost K\"ahler manifolds and prove several dualities and the Lefschetz decomposition for the intersection $\ker\Delta_{\partial}\cap\ker\Delta_{\bar{\partial}}$. Section 4 establishes vanishing theorems for $L^{2}$-harmonic $(p,q)$-forms under the assumption that the K\"ahler form is vanishingly exhaustible. In Section 5, we refine the $L^{2}$-estimates under more concrete geometric conditions, including the case where $\omega$ is $d$(bounded). Section 6 applies these analytic results to compact almost K\"ahler manifolds with negative sectional curvature and small Nijenhuis tensor, proving lower bounds for the Hirzebruch $\chi_{y}$-genus and confirming the Hopf conjecture in this setting.

\section{Almost K\"{a}hler manifold }
We begin by recalling some basic definitions and results concerning differential forms on almost complex and almost Hermitian manifolds.

Let $X$ be a $2n$-dimensional manifold (without boundary) and $J$ be a smooth almost complex structure on $X$. The almost complex structure $J$ acts naturally on the space  $\Om^{k}(X,\C):=\Om^{k}(X)\otimes\C$, inducing a decomposition by type:
	$$\Om^{k}(X,\C)=\bigoplus_{p+q=k}\Om^{p,q}(X),$$
	where  $\Om^{p,q}(X)$ denotes the space of complex forms of type $(p,q)$ with respect to $J$ \cite{HMT}. On $\Om^{p,q}$, the action of $J$ is given by $J(\bullet)=\sqrt{-1}^{q-p}(\bullet)$, which is an isomorphism. If $k$ is even, $J$ also acts on $\Om^{k}(X,\mathbb{C})$ as an involution.  The exterior derivative $d$ acts on these spaces as follows:
	\[d:\Om^{p,q}\rightarrow\Om^{p+2,q-1}\oplus\Om^{p+1,q}\oplus\Om^{p,q+1}\oplus\Om^{p-1,q+2},\]
	and consequently $d$ splits into four components:
	\[d=\mu+\pa+\bar{\pa}+\bar{\mu},\]
	where each component is a graded algebra derivation. The operators $\mu$ and $\bar{\mu}$ are $0$-order differential operators. 
	The bidegrees of the components are given by
	\[|\mu|=(2,-1),\ |\pa|=(1,0),\ |\bar{\pa}|=(0,1),\ |\bar{\mu}|=(-1,2).\]
	Expanding the equation $d^{2}=0$ we obtain the following set of equations:
	\begin{equation}\label{E27}
	\begin{split}
	&\mu^{2}=0,\\
	&\mu\pa+\pa\mu=0,\\
	&\pa^{2}+\mu\bar{\pa}+\bar{\pa}\mu=0,\\
	&\pa\bar{\pa}+\bar{\pa}\pa+\mu\bar{\mu}+\bar{\mu}\mu=0,\\
	&\bar{\pa}^{2}+\bar{\mu}\pa+\pa\bar{\mu}=0,\\
	&\bar{\mu}\bar{\pa}+\bar{\pa}\bar{\mu}=0,\\
	&\bar{\mu}^{2}=0.\\
	\end{split}
	\end{equation} 
The integrability theorem of Newlander and Nirenberg states that the almost complex structure $J$ is integrable if and only if $N_{J}=0$, where
$$N_{J}:TX\otimes TX\rightarrow TX,$$
denotes the Nijenhuis tensor
$$N_{J}(X,Y):=[X,Y]+J[X,JY]+J[JX,Y]-[JX,JY].$$ 
One can show that $\mu+\bar{\mu}$ is equal, up to a scalar, to the dual of the Nijenhuis tensor (cf. \cite{CW20}). In fact,
$$\mu+\bar{\mu}=-\frac{1}{4}(N_{J}\otimes{\rm{id}_{\C} })^{\ast}.$$
where the right hand side has been extended over all forms as a derivation. In particular,  $J$ is integrable if only if $N_{J}=0$, i.e $\mu=0$ \cite{CW20,CW21}.
	
Now, let $(X,J,g)$ be an almost Hermitian manifold of real dimension $2n$, and let $\na$ be the Levi-Civita connection of $(X,g)$. The curvature of $(X,g)$ is defined by 
\[R(e_{i},e_{j})e_{k}=[\na_{e_{i}},\na_{e_{j}}]e_{k}-\na_{[e_{i},e_{j}]}e_{k},\]and the Riemannian curvature tensor is given by $R(e_{i},e_{j},e_{k},e_{l}):=g(R(e_{i},e_{j})e_{k},e_{l})$. 
The associated scalar curvature \(s_g\) and \(\ast\)-scalar curvature \(s_g^*\) are given by
\[s_{g}:=\sum_{i,j=1}^{2n}R(e_{i},e_{j},e_{i},e_{j})\quad and\quad s_{g}^{\ast}:=\sum_{i,j=1}^{2n}R(e_{i},e_{j},Je_{i},Je_{j}).\] 
Let $W$ be the Weyl tensor of $(X,g)$. The following useful formula was proved in \cite{dRS}:
\[(2n-1)s_{g}^{\ast}-s_{g}=2(2n-1)\hat{W}(\w^{\sharp},\w^{\sharp}),\]
where $\hat{W}$ is defined by $\hat{W}(e_{i}\wedge e_{j},e_{k}\wedge e_{l}):=W(e_{i},e_{j},e_{k},e_{l})$ 
and $\w^{\sharp}$ denotes the dual tensor of the fundamental form $\w$.

A fundamental relation between these objects in the almost Kähler setting is given by the following proposition.
\begin{proposition}(\cite[Lemma 2.4]{Sek})\label{Pro4}
	Let $(X,J,\w)$ be a closed $2n$-dimension almost K\"{a}hler manifold. Then the following identities hold:
	\[|\na{J}|^{2}=\frac{1}{4}|N_J|^{2}=2(s_{g}^{\ast}-s_{g})=4\big{(}\hat{W}(\w^{\sharp},\w^{\sharp})-\frac{n-1}{2n-1}s_{g}\big{)}.\]
\end{proposition}

We extended the metric to a complex bilinear form; we now extend the curvature operator at $x\in X$,
 $\mathfrak{R}:\La^{2}T_{x}X\rightarrow\La^{2}T_{x}X$, to a complex linear transformation $\mathfrak{R}^{\C}:
 \La_{2}T_{x}X\otimes\C\rightarrow \La_{2}T_{x}X\otimes\C$.  Following \cite[Lemma 3.3]{Her00} and \cite[Lemma 2.4 ]{Sek} 
 or \cite[Equation (3.1)--(3.3)]{SV}), we have

\begin{lemma}(\cite[Lemma 3.3]{Her00})
Let $Z_{j}$ be any orthonormal basis of $T_{x}^{1,0}$. Then
\[s^{\ast}_{g}-s_{g}=-4\sum_{i,j=1}^{n}(\mathfrak{R}^{\C}(Z_{i}\wedge Z_{j}),\overline{Z_{i}\wedge Z_{j} } ) .\]
\end{lemma}
The Lefschetz operator  $L:\Om^{p,q}\rightarrow\Om^{p+1,q+1}$ is defined by
\[L(\a)=\w\wedge\a.\]
The dual Lefschetz operator $\La$ is the adjoint of $L$, given by $\La=\star^{-1} L\star$ \cite{Huy}. 
The operators \(\bar{\mu}, \bar{\partial}, \partial, \mu \) (each generically denoted by \(\delta\)) each possess an \(L^2\)-adjoint;
we denote the adjoint of \(\delta\) by \(\delta^* \).
When $X$ is closed, one may verify the following relations: \[\bar{\mu}^{\ast}=-\star\mu\star\quad and\quad\bar{\pa}^{\ast}=-\star\pa\star.\]
For each such operator $\de$, the associated Laplacian is defined by
\[\De_{\de}:=\de\de^{\ast}+\de^{\ast}\de.\]
It satisfies
\[\star\De_{\bar{\de}}=\De_{\de}\star,\]
and consequently, the operator $\De_{\bar{\pa}}+\De_{\pa}$ commute with the Hodge star operator:
\[\star(\De_{\bar{\pa}}+\De_{\pa})=(\De_{\bar{\pa}}+\De_{\pa})\star.\]

On an almost K\"{a}hler manifold, Cirici and Wilson \cite{CW20} constructed the almost K\"{a}hler identities, involving the differential operators $\pa$ and $\bar{\pa}$, the operators $\mu$ and $\bar{\mu}$,  the Lefschetz operator $L$, and their complex conjugates and adjoints. We will recall some identities which  will be used in this article.
	
If $A$ and $B$ are operators on forms, defined the graded commutator
	$$[A,B]=AB-(-1)^{\deg A\cdot \deg B}BA,$$
	where $\deg T$ is the integer $l$ for which $$T:\bigoplus_{p+q=k}\Om^{p,q}(X)\rightarrow\bigoplus_{p+q=k+l}\Om^{p,q}(X).$$
	\begin{proposition}(cf. \cite{CW20}) \label{aki}
		For any almost K\"{a}hler manifold the following identities hold:
			\begin{enumerate}
			\item[(1)]  $[\pa,\bar{\pa}^{\ast}]=[\bar{\mu}^{\ast},\bar{\pa}]+[\mu,\pa^{\ast}]$ and $[\bar{\pa},\pa^{\ast}]=[\mu^{\ast},\pa]+[\bar{\mu},\bar{\pa}^{\ast}]$.
			\item[(2)]  $[L,\bar{\mu}^{\ast}]=\sqrt{-1}\mu$, $[L,\mu^{\ast}]=-\sqrt{-1}\bar{\mu}$ and   $[\La,\bar{\mu}]=\sqrt{-1}\mu^{\ast}$, $[\La,\mu]=-\sqrt{-1}\bar{\mu}^{\ast}$.
			\item[(3)] $[L,\bar{\pa}^{\ast}]=-\sqrt{-1}\pa$, $[L,\pa^{\ast}]=\sqrt{-1}\bar{\pa}$ and  $[\La,\bar{\pa}]=-\sqrt{-1}\pa^{\ast}$, $[\La,\pa]=\sqrt{-1}\bar{\pa}^{\ast}$.
	\end{enumerate}
	\end{proposition}
We have one more set of useful relations, which are related to hard Lefschetz duality.
	\begin{corollary}(cf. \cite[Corollary 3.5]{CW20}) \label{C6}
	For any almost K\"{a}hler manifold the following identities hold:
	\begin{enumerate}
		\item[(1)] $[L,\De_{\bar{\pa}}+\De_{\pa}]=0$,
		\item[(2)] $[\La,\De_{\bar{\pa}}+\De_{\pa}]=0$.
	\end{enumerate}
\end{corollary}
\begin{proof}
This conclusion follows directly from \cite[Corollary 3.5]{CW20}.	
\end{proof}

	\section{$L^{2}$-Harmonic Forms and Basic Dualities}

	In this section, we now turn to the core of this article: the $L^{2}$-Hodge theory on complete almost Kähler manifolds. We first define the relevant spaces of harmonic forms and establish their basic properties.
	
	\subsection{de Rham harmonic $(p,q)$-forms}
	
	For any almost Hermitian manifold $(X,J,g)$ of dimension $2n$, there is an
	associated Hodge-star operator 
	\[\star:\Om^{p,q}\rightarrow \Om^{n-q,n-p}\]
	defined by 
	\[\a\wedge\star\bar{\b}=\langle\a,\b\rangle\mathrm{dVol}_{g},\]
	where $\mathrm{dVol}_{g}$ is the volume form determined by the metric $g$. The global $L^{2}$-inner product is defined as
	\[ (u,v)=\int_{X}\langle u,v\rangle dV=\int_{X}u\wedge\star\bar{v},\]
	where $dV=\frac{\w^{n}}{n!}$ is the volume form determined by $\w$. We also write 
	\[|u|^{2}=\langle u,u\rangle, \quad\|u\|_{L^2}^{2}=\int_{X}|u|^{2}dV. \]
	
	Define the space of  $L^{2}$-harmonic forms of bidegree $(p,q)$ as 
$$\mathcal{H}^{p,q}_{(2);d}(X):=\{\a\in\Om^{p,q}_{(2)}(X):\De_{d}\a=0 \},$$
where
$$\Om^{p,q}_{(2)}(X):=\{\a\in\Om^{p,q}(X):\|\a\|_{L^{2}(X)}<\infty\}.$$

\begin{lemma}(cf. \cite[Lemma 3.3]{Hua22})\label{L1}
If an $L^{2}$ (p,q)-form $\a$ on a complete almost Hermitian manifold $X$ is $\De_{d}$-harmonic form, then $d\a=0, d^{\ast}\a=0$.	
\end{lemma}

The $d^{\La}$ operator is related via the Hodge star operator defined with respect to the compatible metric $g$ by the relation, see \cite[Lemma 2.9]{Tseng-Yau},
$$d^{\La}=(-1)^{k+1}\star J^{-1}d\star J^{-1}=-{\star}J^{-1}dJ\star.$$

\begin{lemma}(\cite[Theorem 6.7]{TT} and \cite[Lemma 3.2]{Hua22})\label{L6}
		\begin{equation*}
		\ker d\cap\ker d^{\ast}\cap\Om^{p,q}(X)=\ker d^{\La}\cap\ker d^{\La_{\ast} }\cap\Om^{p,q}(X)
		\end{equation*}
	\end{lemma}
	\begin{proof}
		Noting that $J^{2}=(-1)^{k}$ acting on any $k$-forms with $k=p+q$. We then have 
		$$d\star J^{-1}\a_{p,q}=d\star(-1)^{k}J\a_{p,q}=(-1)^{k}(\sqrt{-1})^{q-p}d\star\a_{p,q},$$
		and
		$$dJ\a_{p,q}=(\sqrt{-1})^{q-p}d\a_{p,q}.$$
		Therefore,
		\begin{equation*}
		|d^{\La}\a_{p,q}|=|J^{-1}d\star J^{-1}\a_{p,q}|=|d\star J^{-1}\a_{p,q}|=|d\star\a_{p,q}|=|d^{\ast}\a_{p,q}|,
		\end{equation*}  
		and
		$$|d^{\La_{\ast}}\a_{p,q}|=|J^{-1}dJ\a_{p,q}|=|d\a_{p,q}|.$$
	Thus $d^{\La}\a_{p,q}=d^{\La_{\ast}}\a_{p,q}=0$ if only if $d^{\ast}\a_{p,q}=d\a_{p,q}=0$. 
	\end{proof}
	
	\begin{proposition}\label{P1}(\cite[Lemma 2.2]{CW20})
		If $\a_{p,q}\in\mathcal{H}^{p,q}_{(2);d}(X)$, then $L(\a_{p,q})=\w\wedge\a_{p,q}\in\mathcal{H}^{p+1,q+1}_{(2);d}(X)$.
	\end{proposition}
	\begin{proof}
		Following Lemma \ref{L1}, we have $d\a_{p,q}=0$ and $d^{\ast}\a_{p,q}=0$. By Lemma \ref{L6}, we then have $d^{\La\ast}\a_{p,q}=0$ and $d^{\La}\a_{p,q}=0$. Using the identities $[d^{\ast},L]=-d^{\La\ast}$ and $d\w=0$, we get $d^{\ast}(\w\wedge\a_{p,q})=0$ and $d(\w\wedge\a_{p,q})=0$.
	\end{proof}

	\subsection{Dolbeault harmonic $(p,q)$-forms}
	Denote by
	$$\mathcal{H}^{p,q}_{(2);\bullet}(X):=\{\a\in\Om^{p,q}_{(2)}(X):\De_{\bullet}\a=0 \} $$
	the space of $L^{2}$ $\De_{\bullet}$-harmonic forms of bidegree $(p,q)$, where $\bullet=\pa,\bar{\pa}$.
	\begin{lemma}(cf. \cite[Lemma3.3]{Hua22})\label{L5}
		If an $L^{2}$ $(p, q)$-form $\a$ on  a complete Hermitian manifold $X$ is $\De_{\pa}$- (resp. $\De_{\bar{\pa}}$-) harmonic form, then $\pa\a=0,\pa^{\ast}\a=0$ (resp. $\bar{\pa}\a=0,\bar{\pa}^{\ast}\a=0$).	
	\end{lemma}
	Following Lemmas \ref{L1} and \ref{L5}, we get
	\begin{corollary}\label{C4}
		If an $L^{2}$ $(p,q)$-form $\a$ on $X$ is $\De_{d}$-harmonic form, then $d\a=0,d^{\ast}\a=0$, i.e. $\pa\a=\bar{\pa}\a=0$, $\pa^{\ast}\a=\bar{\pa}^{\ast}\a=0$ and $\mu\a=\bar{\mu}\a=0$, $\mu^{\ast}\a=\bar{\mu}^{\ast}\a=0$. In particular, $$\mathcal{H}^{p,q}_{(2);d}(X)\subset\mathcal{H}^{p,q}_{(2);\pa}(X)\cap\mathcal{H}^{p,q}_{(2);\bar{\pa}}(X).$$
	\end{corollary}
This corollary shows that $\De_{d}$-harmonic forms are automatically harmonic with respect to both $\pa$ and $\bar{\pa}$, and also annihilate the operators $\mu$ and $\bar{\mu}$. Hence, the space $\mathcal{H}^{d}_{(2);d}(X)$ sits naturally inside the intersection $\ker{\De_{\bar{\pa}}}\cap\ker\De_{\pa}$, which is the main object of study in the following sections.
	

	\begin{proof}[\textbf{Proof of Theorem \ref{T4}}]
	The first duality follows from the identity
	$$\ker(\De_{\pa})\cap\Om^{p,q}=\ker(\De_{\bar{\pa}})\cap\Om^{q,p}
	\Rightarrow\ker(\De_{\pa}+\De_{\bar{\pa}})\cap\Om^{p,q}=\ker(\De_{\pa}+\De_{\bar{\pa}})\cap\Om^{q,p}.$$

	 Hodge duality follows from this same identity together with the relation \[\star(\De_{\bar{\pa}}+\De_{\pa})=(\De_{\bar{\pa}}+\De_{\pa})\star,\]
	 which also proves the Serre duality.	
		
	By Corollary \ref{C6}, we obtain
	$$[L,\De_{\pa}+\De_{\bar{\pa}}]=0,\quad [\La,\De_{\pa}+\De_{\bar{\pa}}]=0,$$
	so both $L$ and $\La$ preserve $\ker{\De_{\pa}}\cap\ker{\De_{\bar{\pa}}}$. Consequently, the maps
	$$L^{n-k}:\ker{\De_{\pa}}\cap\ker{\De_{\bar{\pa}}}\cap\Om^{p,k-p}_{(2)}(X)\stackrel{\cong}
	{\longrightarrow}
	\ker{\De_{\pa}}\cap\ker{\De_{\bar{\pa}}}\cap\Om^{p+n-k,n-p}_{(2)}(X)$$
	are isomorphisms. 
	
	For any $\a^{p,q}\in\ker{\De_{\pa}}\cap\ker{\De_{\bar{\pa}}}\cap\Om^{p,q}_{(2)}(X)$, we can write
	\[\a^{p,q}=\bigoplus_{j\geq\min\{p+q-n,0\} }L^{j}\b^{p-j,q-j},\]
	where $\b^{p-j,q-j}\in\ker\La\cap\Om_{(2)}^{p-j,q-j}$. Therefore,
\[ 0=(\De_{\bar{\pa}}+\De_{\pa})\a= \bigoplus_{0\leq j\leq\min\{p,q\} }L^{j}(\De_{\bar{\pa}}+\De_{\pa})\b^{p-j,q-j},\]
and	
	\[\La(\De_{\bar{\pa}}+\De_{\pa})\b^{p-j,q-j}=(\De_{\bar{\pa}}+\De_{\pa})\La\b^{p-j,q-j}=0\Rightarrow(\De_{\bar{\pa}}+\De_{\pa})\b^{p-j,q-j}\in\ker\La.\]
Hence \[(\De_{\bar{\pa}}+\De_{\pa})\b^{p-j,q-j}=0.\]
This implies that
$$\ker{\De_{\pa}}\cap\ker{\De_{\bar{\pa}}}\cap\Om^{p,q}_{(2)}(X)=\bigoplus_{j\geq\max\{p+q-n,0\}}L^{j}(\ker{\De_{\pa}}\cap\ker{\De_{\bar{\pa}}}\cap P^{p-j,q-j}_{(2)}(X)).$$
	\end{proof}
	\begin{proposition}\label{P3}
		If $\a_{p,q}\in\ker{\De_{\bar{\pa}}}\cap\ker{\De_{\pa}}\cap\Om^{p,q}_{(2)}(X)$, then $L(\a_{p,q})=\w\wedge\a_{p,q}\in\ker{\De_{\bar{\pa}}}\cap\ker{\De_{\pa}}\cap\Om^{p+1,q+1}_{(2)}(X)$.
	\end{proposition}
	\begin{proof}
		It follows from Theorem \ref{T4}.
	\end{proof}

	\section{Vanishing theorem for $L^{2}$-harmonic forms}

	\subsection{Vanishingly exhaustible $k$-form}

	Let $(X,g)$ be a complete Riemannian manifold. Recall that a function $E:X\rightarrow\mathbb{R}$ is called an \textit{exhaustion function} if the sublevel sets
	$$X_{k}=\{x\in X:E(x)<k \}$$
	are relatively compact in $X$ for any $k\in\mathbb{R}$ (cf. \cite{McN03}). In this article, we only consider $C^{1}$ exhaustion function as follows.
	\begin{definition}(\cite[Definition 1]{McN03})
		Let $f:\mathbb{R}\rightarrow\mathbb{R}^{+}$ be continuous and $E$ be a $C^{1}$ exhaustion function. We say that a $k$-form $\theta$ on $X$ is $f(E)$-bounded, if
		$$|\theta(x)|\leq f(E(x)),\ for\ all\ x\in X.$$
	\end{definition}
	Note that the distance function $\rho$ associated to the metric $g$ on a complete manifold $X$ has the property that its differential $d\rho$ is $f(E)$-bounded for $f\equiv\textit{constant}$ and for any exhaustion function $E$. Following the idea of McNeal in \cite{McN03}, we consider some smooth differential forms as following.
	\begin{definition}
	The smooth  $k$-form $\w$, ($k\geq1$), on a complete Riemannian manifold $X$ is vanishingly exhaustible if there exist 
	 $C^{1}$ exhaustion functions $E$ on $X$, continuous, nondecreasing functions $f,g:\mathbb{R}\rightarrow\mathbb{R}^{+}$, 
	 and $C^{1}$ $(k-1)$-forms $\theta$,  on $X$ such that:\\
	 (1) $\w$ is bounded;\\
	 (2) $\w=d\theta$;\\
	(3) $\theta$ is $f(E)$-bounded, $dE$ is $g(E)$-bounded;\\
	(4) the series $$\sum_{k=N}^{\infty}\frac{1}{f(k)g(k)}$$ diverges.
	\end{definition}

	Returning to the almost Kähler setting, we extend the concept of a vanishingly exhaustible form by allowing a finite linear combination of such forms to approximate the Kähler form $\w$ in the $L^{\infty}$-norm. The constant $c$ quantifies the quality of the approximation; smaller values of $c$ correspond to better approximations.
	\begin{definition}\label{D1}
		The complete almost K\"{a}hler manifold $(X,J,\w)$ is $(L,c)$-vanishingly exhaustible if there exist a sequence of vanishing  exhaustible $2$-forms, $\{\w_{1},\cdots,\w_{L}\}$ and a uniform positive constant $c$ such that
		\begin{equation*}
		\sup|\w-\sum_{i=1}^{L}\w_{i}^{1,1}|\leq c,
		\end{equation*}
		where $\w_{i}^{1,1}$ is the $(1,1)$-part of $\w_{i}$.
	\end{definition}
	\begin{example}
		(1)	Let $(X,g)$ be a simply-connected $n$-dimensional complete Riemannian manifold with sectional curvature bounded from above by a negative constant, i.e.
		$$\sec\leq-K$$
		for some $K>0$. We have the following classical fact pointed by Gromov \cite{Gro} (one also can see \cite[Proposition 8.4]{Bal} and \cite[Lemma 3.2]{CY}).
		
		For any bounded and closed $k$-form $\w$ on $X$, where $k>1$, there exists a bounded $(k-1)$-form $\theta$ on $X$ such that $\w=d\theta$ and 
		$$\|\theta\|_{L^{\infty}(X)}\leq K^{-\frac{1}{2}}\|\w\|_{L^{\infty}(X)}.$$
		Following above statement, we only need take $f(x)=g(x)=E(x)=1$.

		(2) The author in \cite{McN02} considered a complete K\"{a}hler manifold $(X,\w)$ which given by a global potential, i.e.,
		 $$\w=\sqrt{-1}\pa\bar{\pa} f=\frac{1}{2}dJd f$$
		for some smooth function $f$. The hypotheses of the function $f$ on Theorem 2.6 of \cite{McN02} follows from the Definition \ref{D1} 
		by taking $f(x)=g(x)=\sqrt{A+Bx}$ and $E(x)=f(x)$. The hypotheses of \cite[Theorem 2]{CF} similarly follows from Definition \ref{D1} by taking $f(x)=c(1+x)$, $g(x)=1$, and $E(x)=\rho(x,x_{0})$ where $\rho(x,x_{0})$ denotes the Riemannian distance between $x$ and a fixed base point $x_{0}\in X$. 
		
		(3) Let $(X,g)$ be a complete manifold of finite volume with pinched negative sectional curvature, i.e.,
		 there are positive constants $a,b$ such that $-b^{2}\leq \sec\leq -a^{2}$. 
		 We now  recall some standard facts about the topology and geometry of this manifold (see \cite{Ebe,HI,Yeg}). 
		 At First, $X$ has a finite number of ends, it means that $X=X_{0}\cup E_{i}$, 
		 where $X_{0}$ is a compact manifold with boundary $\pa X_{0}$, and the boundaries of $E_{i}$, $\pa E_{i}$, 
		 are the components of $\pa X_{0}$. For each end $E_{i}$, it is $C^{2}$-diffeomorphic to $\mathbb{R}\times\pa E_{i}$, 
		 and the metric on $E_{i}$ is as follows:
        $$
		g=d^{2}\rho_{i}+h_{\rho_{i}},
		$$
		where $\rho_{i}$ is the Busemann function, $h_{\rho_{i}}$ is a family of metrics on the compact manifold $\pa E_{i}$, and satisfies
		$e^{-b\rho_{i}}h_{0}\leq h_{\rho_{i}}\leq e^{-a\rho_{i}}h_{0}$.
		
		On a complete almost K\"{a}hler manifold of finite volume and with pinched negative sectional curvature,
		 the K\"{a}hler form $\w$ cannot be $d$(bounded). In fact, for any bounded nonzero harmonic smooth $k$-form $\a$, $k\geq1$,
		  on a complete Riemannian manifold of finite volume, $\a$ cannot be $d$(bounded). 
		  If not, there exists a $(k-1)$-form $\b$ such that $\a=d\b$ and $\b$ is bounded, then $d^{\ast}d\b=0$. 
		  As the volume is finite, one can see that $\b\in L^{2}$. Therefore, 
		  $$0=(d^{\ast}d\b,\b)=\|d\b\|_{L^2}^{2},$$ 
		  i.e., $\a=d\b=0$. 
		  However, we have the following result. 

		\begin{proposition}(cf. \cite[Lemma 3.2]{Yeg})
			Let $(X,J,\w)$ be a complete almost K\"{a}hler manifold of finite volume and with pinched negative curvature $-b^{2}\leq \sec\leq -a^{2}<0$. 
			Then outside a compact subset, its K\"{a}hler form is $d$(bounded). 
			More specifically, there exist a bounded open subset $D\subset X$ and a bounded and continuous $1$-form $\theta$ 
			such that we have  $\w=d\theta$ in the weak sense on $X\backslash D$. 
		\end{proposition}
	\begin{proof}
	The proof is the same as \cite[Lemma 3.2]{Yeg}.
\end{proof}
	\end{example}

	\subsection{Key lemma}
We now prove a lemma extending Stokes formula to complete manifolds under suitable conditions.
	\begin{lemma}\label{L8}
		Let  $X$ be a complete Riemannian manifold. Suppose that $\w$ is a vanishingly exhaustible $k$-form with $k\geq1$. Then for any $\a\in\Om^{k+l}_{(2)}(X)\cap\ker d^{\ast}$, and $\b\in\Om^{l}_{(2)}(X)\cap\ker d$, we have
		\begin{equation*}
		(\a, \w\wedge \b)=0.
		\end{equation*}
	\end{lemma}
	\begin{proof}
		Let $h:\mathbb{R}\rightarrow\mathbb{R}$ be smooth, $0\leq h\leq1$,
		$$
		h(t)=\left\{
		\begin{aligned}
		1, &  & t\geq 1, \\
		0, &  & t\leq 0
		\end{aligned}
		\right.
		$$
		and consider the compactly supported function
		$$h_{k}(x)=h(k-E(x)),$$
		where $k$ is a positive integer. Note that ${\rm{supp}}h_{k}\subset X_{k}$ and $h_{k}=1$ on $X_{k-1}$.
		
		We denote $\w:=d\theta$. Let $\gamma=\theta\wedge\b$. Since $h_{k}\cdot\gamma$ has compact support, 
		and $d^{\ast}\a=0$, an integration by parts gives
		\begin{equation} \label{E1}
		(\a,d(h_{k}\cdot\gamma))=(d^{\ast}\a,h_{k}\cdot\gamma )=0.
		\end{equation}
		Since $d\b=0$, we also have
		\begin{equation} \label{E2}
		d(h_{k}\cdot\gamma)=-h'(k-E)\cdot dE\wedge\theta\wedge\b+h_{k}\cdot\w\wedge\b.
		\end{equation}
		We now substitute (\ref{E2}) into (\ref{E1}) and consider the two terms coming from the right-hand side of (\ref{E2}) separately. For the first term, the fact that $supp h'_{k}\subset X_{k}\backslash X_{k-1}$ and the fact that $\w$ is bounded imply 
		\begin{equation}\label{E3}
		\begin{split}
		|(\a,h'_k\cdot dE\wedge\theta\wedge\b)|&\leq\int_{ X_{k}\backslash X_{k-1}}|dE\wedge\theta|\cdot|\a|\cdot|\b|\\
		&\leq\int_{ X_{k}\backslash X_{k-1}}(f(E)g(E))\cdot|\a|\cdot|\b|\\
		&\leq (f(k)\cdot g(k))\int_{ X_{k}\backslash X_{k-1}}|\a|\cdot|\b|,\\
		\end{split}
		\end{equation}
		for the functions $f,g$ in Definition \ref{D1}. The second inequality follows from our hypotheses on $E$ and $\theta$. The assumption that $\a,\b\in L^{2}$ implies that $|\a|\cdot|\b|\in L^{1}$, then there exists a subsequence $\{k_{i}\}$ such that
		\begin{equation}\label{E4}
		(f(k_{i})\cdot g(k_{i}))\int_{ X_{k_{i}}\backslash X_{k_{i}-1}}|\a|\cdot|\b|\rightarrow0,\ as\ i\rightarrow\infty.
		\end{equation}
		Otherwise, for some $c>0$, we have
		\begin{equation*}
		\begin{split}
		\int_{X}|\a|\cdot|\b|&=\sum_{k=1}^{\infty}\int_{ X_{k}\backslash X_{k-1}}|\a|\cdot|\b|\\
		&\geq c\sum_{k=1}^{\infty}\frac{1}{f(k)g(k)}
		=\infty,
		\end{split}
		\end{equation*}
		which leads to a contradiction.
		
		So, for the sequence gives by (\ref{E4}), it follows from (\ref{E3}) that
		\begin{equation}\label{E5}
		\lim_{i\rightarrow\infty}(\a,h'_{k_{i}}\cdot dE\wedge\theta\wedge\b )=0.
		\end{equation}
		For the term coming from the second term on the right-hand side of (\ref{E2}), the dominated convergences theorem implies 
		\begin{equation}\label{E6}
		\lim_{k\rightarrow\infty}(\a,h_{k}\cdot\w\wedge\b)=(\a,\w\wedge\b).
		\end{equation}
		Substituting (\ref{E5}) and (\ref{E6}) into (\ref{E1}), it follows that $(\a,\w\wedge\b)=0$. 
	\end{proof}

We observe that exact (1,1)-forms have another expression.
	\begin{proposition}\label{P5}
		If the $(1,1)$-form $\tilde{\w}$ on a complete almost K\"{a}hler manifold $X$ is exact, then there exists a $1$-form $\theta$ such that
		\begin{equation*}
		\tilde{\w}=\pa\theta^{0,1}+\bar{\pa}\theta^{1,0},
		\end{equation*}
		where $\theta^{1,0}$ (resp. $\theta^{0,1}$) is the $(1,0)$ (resp.$(0,1)$) part of $\theta$. 
	\end{proposition}
	\begin{proof}
				By the hypothesis, we have $\tilde{\w}=d\theta$. Decomposing $d$ and $\theta$ according to type yields
		\begin{equation*}
		\begin{split}
		d\theta
		&=(\pa+\mu+\bar{\pa}+\bar{\mu})(\theta^{1,0}+\theta^{0,1})\\
		&=(\pa\theta^{1,0}+\mu\theta^{0,1})+(\pa\theta^{0,1}+\bar{\pa}\theta^{1,0})+(\bar{\pa}\theta^{0,1}+\bar{\mu}\theta^{1,0}).\\
		\end{split}
		\end{equation*}
		Since $\tilde{\w}$ is a $(1,1)$-form, the components of types $(2,0)$ and $(0,2)$ must vanish. Hence
		\[\pa\theta^{1,0}+\mu\theta^{0,1}=\bar{\pa}\theta^{0,1}+\bar{\mu}\theta^{1,0}=0\]
		and consequently
		\[\tilde{\w}=\pa\theta^{0,1}+\bar{\pa}\theta^{1,0}. \]
		This completes the proof.
	\end{proof}
	Following Proposition \ref{P5} and Lemma \ref{L8}, we then have
	\begin{lemma}\label{L3}\label{L4}
		Let  $(X,J,\w)$ be a complete almost K\"{a}her manifold. Suppose that $\tilde{\w}$ is a vanishingly exhaustible $(1,1)$-form. For any $\a\in\Om^{p,q}_{(2)}(X)\cap\ker \pa^{\ast}\cap\ker\bar{\pa}^{\ast}$, and $\b\in\Om^{p-1,q-1}_{(2)}(X)\cap\ker\pa\cap\ker\bar{\pa}$, we have
		\begin{equation*}
		(\a, \tilde{\w}\wedge \b)=0.
		\end{equation*}
	\end{lemma}
	\begin{proof}
		Following Proposition \ref{P5},	 we have $\tilde{\w}=\pa\theta^{0,1}+\bar{\pa}\theta^{1,0}$. Let $\gamma_{1}=\theta^{1,0}\wedge\b$ and $\gamma_{2}=\theta^{0,1}\wedge\b$. We denote by $h_{k}$ the compactly supported function on Lemma \ref{L8}. Noting  that 
		\begin{equation*}
		(h'(k-E)\cdot dE\wedge\theta\wedge\b)^{p,q}=h'(k-E)\cdot \bar{\pa}E\wedge\theta^{1,0}\wedge\b+h'(k-E)\cdot\pa E\wedge\theta^{0,1}\wedge\b.
		\end{equation*}
		Since $h_{k}\cdot\gamma$ has compact support, and $\pa^{\ast}\a=\bar{\pa}^{\ast}\a=0$, an integration by parts gives
		\begin{equation}\label{E33}
		\begin{split}
		0&=(\bar{\pa}^{\ast}\a,h_{k}\cdot\gamma_{1})+(\pa^{\ast}\a,h_{k}\cdot\gamma_{2})\\
		&=(\a,\bar{\pa}(h_{k}\cdot\gamma_{1})+\pa(h_{k}\cdot\gamma_{2}) )\\
		&=(\a,-h'(k-E)\cdot \bar{\pa}E\wedge\theta^{1,0}\wedge\b+h_{k}\cdot\bar{\pa}\theta^{1,0}\wedge\b)\\
		&+(\a,-h'(k-E)\cdot\pa E\wedge\theta^{0,1}\wedge\b+h_{k}\cdot\pa\theta^{0,1}\wedge\b)\\
		&=(\a,-h'(k-E)\cdot dE\wedge\theta\wedge\b)+(\a, h_{k}\cdot\tilde{\w}\wedge\b).\\
		\end{split}
		\end{equation}
		Following the idea in Lemma \ref{L8}, there exists a subsequence $\{k_{i}\}$ such that 
		\begin{equation}\label{E34}
		\lim_{i\rightarrow\infty}(\a,h'_{k_{i}}\cdot dE\wedge\theta\wedge\b )=0.
		\end{equation}
		and
		\begin{equation}\label{E35}
		\lim_{k\rightarrow\infty}(\a,h_{k}\cdot\tilde{\w}\wedge\b)=(\a,\tilde{\w}\wedge\b).
		\end{equation}
		Substituting (\ref{E34}) and (\ref{E35}) into (\ref{E33}), it follows that $(\a,\tilde{\w}\wedge\b)=0$. 
	\end{proof}
Lemma \ref{L4} is the technical heart of the vanishing theorem. 
It shows that the $(1,1)$-part of an exact form $\tilde{\w}$ is orthogonal to certain harmonic forms, p
rovided $\tilde{\w}$ is vanishingly exhaustible. This orthogonality is precisely what allows us to compare $\|\w\wedge\a\|_{L^2}$ the perturbation term.

\subsection{Vanishing theorem}
	We now establish the vanishing theorem for $L^{2}$-harmonic forms on complete almost K\"{a}hler manifold.  First, following Lemma \ref{L4}, we have an estimate on ($\bar{\pa}$, $\pa$)-harmonic $(p,q)$-forms as follows.
	
	\begin{corollary}\label{C1}
		Let  $(X,J,\w)$ be a complete almost K\"{a}hler manifold. Suppose that $\{\w_{i}\}$, $i=1,\cdots,L$ is a sequence of vanishing exhaustible $2$-forms. Then for any $\a\in\ker{\De_{\bar{\pa}}}\cap\ker{\De_{\pa}}\cap\Om^{p,q}_{(2)}(X)$,  we have
		\begin{equation*}
		\|\w\wedge\a\|_{L^2}\leq \|(\w-\sum_{i=1}^{L}\w^{1,1}_{i})\wedge \a\|_{L^2}.
		\end{equation*}
	\end{corollary}
	
	\begin{proof}
		Suppose that $\a\in\ker{\De_{\bar{\pa}}}\cap\ker{\De_{\pa}}\cap\Om^{p,q}_{(2)}(X)$. By Proposition \ref{P3},
		 $\w\wedge\a\in\ker{\De_{\bar{\pa}}}\cap\ker{\De_{\pa}}\cap\Om^{p+1,q+1}_{(2)}(X)$ and so Lemma \ref{L5} implies that $$\pa^{\ast}(\w\wedge\a)=\bar{\pa}^{\ast}(\w\wedge\a)=0.$$ 
		Noting that $\w_{i}^{(0,2)}\wedge\a\in\Om^{p,q+2}(X)$, $\w_{i}^{(2,0)}\wedge\a\in\Om^{p+2,q}(X)$ and $\w\wedge\a\in\Om^{p+1,q+1}(X)$. Therefore, we have
		\begin{equation*}
		(\w\wedge\a,\w_{i}^{1,1}\wedge\a)=(\w\wedge\a,\w_{i}\wedge\a).
		\end{equation*}
		Following Lemma \ref{L3}, we obtain
		\begin{equation}\label{E22}
		(\w\wedge\a,\w_{i}\wedge \a)=0.
		\end{equation}	
		By (\ref{E22}), one can see that
		\begin{equation*}
		\begin{split}
		\|\w\wedge\a\|_{L^2}^{2}&=(\w\wedge\a, \sum_{i=1}^{L}\w^{1,1}_{i}\wedge\a)+(\w\wedge\a,(\w-\sum_{i=1}^{L}\w^{1,1}_{i})\wedge \a)\\
		&=(\w\wedge\a,(\w-\sum_{i=1}^{L}\w^{1,1}_{i})\wedge \a).\\
		\end{split}
		\end{equation*}
		Hence,
		$$\|\w\wedge\a\|_{L^2}\leq \|(\w-\sum_{i=1}^{L}\w^{1,1}_{i})\wedge \a\|_{L^2}.$$
	\end{proof}
	This corollary plays a crucial role in the vanishing argument. It shows that the $L^{2}$-norm of $\w\wedge\a$ can be controlled by the perturbation of $\w$ away from a sum of vanishingly exhaustible forms. This estimate is the key to proving Theorem \ref{T2}
	\begin{proof}[\textbf{Proof of Theorem \ref{T2}}]
		By the hypotheses, there exists a sequence of $(1,1)$-forms $\{\w_{1},\cdots,\w_{L}\}$ such that 
		$$\sup|\w-\sum_{i=1}^{L}(\w_{i})^{1,1}|\leq c.$$	
		Suppose that $p+q<n$ and $\a\in\ker{\De_{\bar{\pa}}}\cap\ker{\De_{\pa}}\cap\Om^{p,q}_{(2)}(X)$. By Corollary \ref{C1}, we have
		\begin{equation*}
		\begin{split}
		\|\w\wedge\a\|_{L^2}&\leq\|(\w-\sum_{i=1}^{L}\w^{1,1}_{i})\wedge\a)\|_{L^2}\\
		&\leq C_{1}(p,q)\sup|\w-\sum_{i=1}^{L}\w^{1,1}_{i}|\cdot\|\a\|_{L^2}.\\
		\end{split}
		\end{equation*}
		Here we use the inequality
		$$\langle\a\wedge\b,\a\wedge\b\rangle\leq\binom{r+s}{r}\langle\a,\a\rangle\langle\b,\b\rangle,$$
		where $\a\in\Om^{r}$ and $\b\in\Om^{s}$. Following \cite[Corollary 1.2.28]{Huy}, for any $\a\in\Om^{k}(X)$ we have
		$$[\La,L]\a=(n-k)\a.$$ 
		Therefore,
		$$
		(\w\wedge\a,\w\wedge\a)=([\La,L]\a+\w\wedge(\La\a),\a)=(n-k)\|\a\|_{L^2}^{2}+\|\La\a\|_{L^2}^{2}.
		$$
		We then have
		$$\|\a\|_{L^2}\leq C_{2}(p,q,n)\|\w\wedge\a\|_{L^2}.$$
		
		Combining the preceding inequalities yields
		$$\|\a\|_{L^2}\leq C_{1}C_{2}\sup|\w-\sum_{i=1}^{L}\w^{1,1}_{i}|\cdot\|\a\|_{L^2}\leq cC_{1}C_{2}\|\a\|_{L^2}.$$
		We can choose $c$ small enough to such that $cC_{1}C_{2}<1$, then $\a=0$.  Finally, Poincar\'{e} duality extends the argument just given to the case when $p+q>n$. Following Corollary \ref{C4}, it's easy to see $$\mathcal{H}^{p,q}_{(2);d}(X)=0$$
		for all $p+q\neq n$. 
	\end{proof}

	\section{The $L^{2}$-estimates}

	Throughout this section, we write $\a\lesssim\b$ to mean that $\a\leq C\b$ for some positive constant $C$ independent of certain parameters on which $\a$ and $\b$ depend. The parameters on which $C$ is independent will be clear or specified at each occurrence. We also use $\b\lesssim\a$ and $\a\approx\b$ analogously. We also denote by $\Om^{k}_{0}(X)$ (resp. $\Om^{p,q}_{0}(X)$) the smooth $k$- (resp. $(p, q)$-) forms with compact support on $X$.

	A differential form $\a$ in a Riemannian manifold $(X,g)$ is called bounded with respect to the metric $g$ if the $L^{\infty}$-norm of $\a$ is finite, namely,
	$$\|\a\|_{L^{\infty}(X)}=\sup_{x\in X}|\a(x)|<\infty.$$
	By definition, a $k$-form $\a$ is said to be $d$(bounded) if $\a=d\b$, where $\b$ is a bounded $(k-1)$-form. It is obvious that if $X$ is compact, then every exact form is $d$(bounded).  
	\begin{proposition} 
		Let $(X,J,\w)$ be a complete $2n$-dimensional almost K\"{a}hler manifold. If $\theta$ is a bounded $1$-form, then for any $\a\in\Om^{k}_{0}(X)$, $(k<n)$,
		\begin{equation}\label{E23}
		|((d\theta)^{1,1}\wedge\a,\w\wedge\a)|\leq c(n,k)\|\theta\|_{L^{\infty}(X)}\|\a\|_{L^{2}(X)}((\De_{\pa}+\De_{\bar{\pa}})\a,\a)^{\frac{1}{2}},
		\end{equation} 
		where $c(n,k)$ is a positive constant which depends only $n,k$.
	\end{proposition}
	\begin{proof}
		An integration by parts gives
		\begin{equation*}
		\begin{split}
		((d\theta)^{1,1}\wedge\a,\w\wedge\a)&=((\pa\theta^{0,1}+\bar{\pa}\theta^{1,0})\wedge\a,\w\wedge\a )\\
		&=(\pa(\theta^{0,1}\wedge\a)+\theta^{0,1}\wedge\pa\a,\w\wedge\a)\\
		&+(\bar{\pa}(\theta^{1,0}\wedge\a)+\theta^{1,0}\wedge\bar{\pa}\a,\w\wedge\a) \\
		&=(\theta^{0,1}\wedge\a, [\pa^{\ast},L]\a)+(\theta^{0,1}\wedge\a, L(\pa^{\ast}\a)) +(\theta^{0,1}\wedge\pa\a,\w\wedge\a)\\
		&+(\theta^{1,0}\wedge\a,[\bar{\pa}^{\ast},L]\a)+(\theta^{1,0}\wedge\a,L(\bar{\pa}^{\ast}\a))+(\theta^{1,0}\wedge\bar{\pa}\a,\w\wedge\a) \\
		&=(\theta^{0,1}\wedge\a, -\sqrt{-1}\bar{\pa}\a)+(\theta^{0,1}\wedge\a, L(\pa^{\ast}\a)) +(\theta^{0,1}\wedge\pa\a,\w\wedge\a)\\
		&+(\theta^{1,0}\wedge\a,\sqrt{-1}\pa\a)+(\theta^{1,0}\wedge\a,L(\bar{\pa}^{\ast}\a))+(\theta^{1,0}\wedge\bar{\pa}\a,\w\wedge\a),\\
		\end{split}
		\end{equation*}
		where we used the almost K\"{a}hler identities $ [\pa^{\ast},L]=-\sqrt{-1}\bar{\pa}$, $[\bar{\pa}^{\ast},L]=\sqrt{-1}\pa$ in Proposition \ref{aki}. Therefore, we get 
		\begin{equation*}
		\begin{split}
		&\quad((d\theta)^{1,1}\wedge\a,\w\wedge\a)\\
		&\lesssim \|\theta^{0,1}\|_{L^{\infty}(X)}\|\a\|_{L^{2}(X)}\|\bar{\pa}\a\|_{L^{2}(X)}+\|\theta^{0,1}\|_{L^{\infty}(X)}\|\a\|_{L^{2}(X)}\|\pa^{\ast}\a\|_{L^{2}(X)}+\|\theta^{0,1}\|_{L^{\infty}(X)}\|\a\|_{L^{2}(X)}\|\pa\a\|_{L^{2}(X)} \\
		&+\|\theta^{1,0}\|_{L^{\infty}(X)}\|\a\|_{L^{2}(X)}\|\pa\a\|_{L^{2}(X)}+\|\theta^{1,0}\|_{L^{\infty}(X)}\|\a\|_{L^{2}(X)}\|\bar{\pa}^{\ast}\a\|_{L^{2}(X)}+\|\theta^{1,0}\|_{L^{\infty}(X)}\|\a\|_{L^{2}(X)}\|\bar{\pa}\a\|_{L^{2}(X)}\\
		&\leq c(n,k)\|\theta\|_{L^{\infty}(X)}\|\a\|_{L^{2}(X)}((\De_{\pa}+\De_{\bar{\pa}})\a,\a)^{\frac{1}{2}}.
		\end{split}
		\end{equation*}
	\end{proof}

	\begin{theorem}\label{T5}
		Let $(X,J,\w)$ be a complete $2n$-dimensional almost K\"{a}hler manifold. Suppose that there exists a bounded $1$-form $\theta$ such that
		$$\sup|\w-(d\theta)^{1,1}|\leq c,$$
		then for any $\a\in\Om^{k}_{0}(X)$, $(k\neq n)$,
		\begin{equation}\label{E24}
		\|\a\|_{L^2}(1-c(n,k)c)\leq c(n,k)\|\theta\|_{L^{\infty}(X)}((\De_{\pa}+\De_{\bar{\pa}})\a,\a)^{\frac{1}{2}}.
		\end{equation} 
		Furthermore, if $c$ is small enough such $c(n,k)c<1$, then 
		$$\ker(\De_{\pa}+\De_{\bar{\pa}})\cap\Om^{k}_{(2)}(X)=\{0\}.$$
	\end{theorem}

	\begin{proof}
		Suppose that $k<n$ and $\a\in\Om^{k}_{0}(X)$. Following (\ref{E23}), we get
		\begin{equation*}
		\begin{split}
		\|\a\|_{L^2}^{2}&\lesssim \|\w\wedge\a\|_{L^2}^{2}\\
		&=(\w\wedge\a, (d\theta)^{1,1}\wedge\a)+(\w\wedge\a,(\w-(d\theta)^{1,1})\wedge \a)\\
		&\leq c(n,k)c\|\a\|_{L^2}^{2}+c(n,k)\|\theta\|_{L^{\infty}(X)}\|\a\|_{L^{2}(X)}((\De_{\pa}+\De_{\bar{\pa}})\a,\a)^{\frac{1}{2}}.\\
		\end{split}
		\end{equation*}
		Rearrangement gives (\ref{E24}). The case $k>n$ follows by the Poincar\'{e} duality as the operator $\star:\Om^{p,q}\rightarrow\Om^{n-q,n-p}$ commutes with $\De_{\bar{\pa}}+\De_{\pa}$.
	\end{proof}

	\begin{corollary}\label{C2}
		Let $(X,J,\w)$ be a complete $2n$-dimensional almost K\"{a}hler manifold. Suppose that there exists a bounded $1$-form $\theta$ such that
		$$\sup|\w-(d\theta)^{1,1}|\leq c,$$
		then for any $\a\in\Om^{k}_{0}(X)$, $(k\neq n)$,
		\begin{equation}\label{E25}
		\|\a\|_{L^2}^{2}(1-c(n,k))(1-c(n,k)c-4c(n,k)\|\theta\|_{L^{\infty}(X)}\sup|N_{J}|)\leq c^{2}(n,k)\|\theta\|^{2}_{L^{\infty}(X)}(\De_{d}\a,\a).
		\end{equation} 
		Furthermore, if $c(n,k)(c+4\|\theta\|_{L^{\infty}(X)}\sup|N_{J}|)<1$, then for any $k\neq n$,
		$$\mathcal{H}^{k}_{(2);d}(X)=\{0\}.$$
	\end{corollary}
	\begin{proof}
		Firstly, expanding $\De_{d}=[d,d^{\ast}]$ and using $d=\pa+\mu+\bar{\pa}+\bar{\mu}$, we have
		\begin{equation*}
		\begin{split}
		\De_{d}&=\De_{\pa}+\De_{\bar{\pa}}+\De_{\mu}+\De_{\bar{\mu}}\\
		&+[\bar{\pa},\pa^{\ast}]+[\pa,\bar{\pa}^{\ast}]\\
		&+[\pa+\bar{\pa},\mu^{\ast}+\bar{\mu}^{\ast}]+[\mu+\bar{\mu},\pa^{\ast}+\bar{\pa}^{\ast}].\\
		\end{split}
		\end{equation*}
		We observe that
		\begin{equation*}
		\begin{split}
		\mathrm{I}&=(([\bar{\pa},\pa^{\ast}]+[\pa,\bar{\pa}^{\ast}])\a,\a)\\
		&=2([\mu^{\ast},\pa]\a,\a)+2([\bar{\mu},\bar{\pa}^{\ast}]\a,\a)\\
		&=2(\pa\a,\mu\a)+2(\mu^{\ast}\a,\pa^{\ast}\a)+2(\bar{\pa}^{\ast}\a,\a)+2(\bar{\mu}\a,\bar{\pa}\a).
		\end{split}
		\end{equation*}
		Here we use the identities in \cite[Proposition 3.3]{CW20} as follows
		$$[\pa,\bar{\pa}^{\ast}]=[\bar{\mu}^{\ast},\bar{\pa}]+[\mu,\pa^{\ast}]\ and\ [\bar{\pa},\pa^{\ast}]=[\mu^{\ast},\pa]+[\bar{\mu},\bar{\pa}^{\ast}].$$
		Therefore, we have
		\begin{equation*}
		\begin{split}
		|\mathrm{I}|&\leq 2\sup|N_{J}|\cdot\|\a\|_{L^2}(\|\pa\a\|_{L^2}+\|\pa^{\ast}\a\|_{L^2}+\|\bar{\pa}\a\|_{L^2}+\|\bar{\pa}^{\ast}\a\|_{L^2})\\
		&\leq 2\sup|N_{J}|\cdot\|\a\|_{L^2}((\De_{\pa}+\De_{\bar{\pa}})\a,\a)^{\frac{1}{2}}.\\
		\end{split}
		\end{equation*}
		We also observe that
		\begin{equation*}
		\begin{split}
		|\mathrm{II}|&=|(([\pa+\bar{\pa},\mu^{\ast}+\bar{\mu}^{\ast}]+[\mu+\bar{\mu},\pa^{\ast}+\bar{\pa}^{\ast}])\a,\a)|\\
		&=2|((\mu^{\ast}+\bar{\mu}^{\ast})\a,(\pa^{\ast}+\bar{\pa}^{\ast})\a)+((\mu+\bar{\mu})\a,(\pa+\bar{\pa})\a)|\\
		&\leq 2\sup|N_{J}|\cdot\|\a\|_{L^2}(\|\pa\a\|_{L^2}+\|\pa^{\ast}\a\|_{L^2}+\|\bar{\pa}\a\|_{L^2}+\|\bar{\pa}^{\ast}\a\|_{L^2})\\
		&\leq 2\sup|N_{J}|\cdot\|\a\|_{L^2}((\De_{\pa}+\De_{\bar{\pa}})\a,\a)^{\frac{1}{2}}.\\
		\end{split}
		\end{equation*}
		Combining the preceding inequalities with estimate (\ref{E24}) yields
		\begin{equation*}
		\begin{split}
		(\De_{d}\a,\a)&\geq((\De_{\pa}+\De_{\bar{\pa}})\a,\a)-4\sup|N_{J}|\cdot\|\a\|_{L^2}((\De_{\pa}+\De_{\bar{\pa}})\a,\a)^{\frac{1}{2}}.\\
		&\geq((\De_{\pa}+\De_{\bar{\pa}})\a,\a)^{\frac{1}{2}}(c(n,k)^{-1}\|\theta\|^{-1}_{L^{\infty}(X)}(1-c(n,k)c)-4\sup|N_{J}|)\|\a\|_{L^2} \\
		&\geq c(n,k)^{-1}\|\theta\|^{-1}_{L^{\infty}(X)}(1-c(n,k)c)(c(n,k)^{-1}\|\theta\|^{-1}_{L^{\infty}(X)}(1-c(n,k)c)-4\sup|N_{J}|)\|\a\|_{L^2}^{2}.
		\end{split}
		\end{equation*} 
		Rearrangement gives (\ref{E25}).
	\end{proof}

	Recall that by \cite{CW20} (cf. also \cite{TT}) on complete almost-K\"{a}hler manifolds,
	$$\De_{\bar{\pa}}+\De_{\mu}=\De_{\pa}+\De_{\bar{\mu}}.$$
	A sharper estimate for the individual Laplacians $\De_{\pa}$ and $\De_{\bar{\pa}}$ follows from Corollary \ref{C3}.
	
	\begin{corollary}\label{C3}
		Let $(X,J,\w)$ be a complete $2n$-dimensional almost K\"{a}hler manifold. Suppose that there exists a bounded $1$-form $\theta$ such that
		$$\sup|\w-(d\theta)^{1,1}|\leq c,$$
		then for any $\a\in\Om^{k}_{0}(X)$, $(k\neq n)$,
		\begin{equation}\label{E26}
		\|\a\|_{L^2}^{2}((1-c(n,k)c)^{2}-2c(n,k)^{2}\|\theta\|^{2}_{L^{\infty}(X)}\sup|N_{J}|^{2})\leq 2c^{2}(n,k)\|\theta\|^{2}_{L^{\infty}(X)}(\De_{\bullet}\a,\a),
		\end{equation} 
		where $\bullet=\pa,\bar{\pa}$. Furthermore, if $c(n,k)(c+\sqrt{2}\|\theta\|_{L^{\infty}(X)}\sup|N_{J}|)<1$, then for any $k\neq n$,
		$$\mathcal{H}^{k}_{(2);\bullet}(X)=\{0\}.$$
	\end{corollary}

	\begin{proof}
		For any  $\a\in\Om^{k}_{0}(X)$, we have
		\begin{equation*}
		\begin{split}
		((\De_{\mu}+\De_{\bar{\mu}})\a,\a)&=\|\mu\a\|_{L^2}^{2}+\|\bar{\mu}\a\|_{L^2}^2+\|\mu^{\ast}\a\|_{L^2}^{2}+\|\bar{\mu}^{\ast}\a\|^2_{L^2}\\
		&\leq 2\sup|N_{J}|^{2}\cdot\|\a\|_{L^2}^{2}.\\
		\end{split}
		\end{equation*}
		We observe that
		\begin{equation*}
		\begin{split}
		&\De_{\pa}=\frac{1}{2}(\De_{\pa}+\De_{\bar{\pa}}+\De_{\mu}-\De_{\bar{\mu}}),\\
		&\De_{\bar{\pa}}=\frac{1}{2}(\De_{\pa}+\De_{\bar{\pa}}-\De_{\mu}+\De_{\bar{\mu}}).\\
		\end{split}
		\end{equation*}
		Combining the preceding inequalities with estimate (\ref{E24}) yields
		\begin{equation*}
		\begin{split}
		(\De_{\bullet}\a,\a)&\geq\frac{1}{2}((\De_{\pa}+\De_{\bar{\pa}})\a,\a)-\sup|N_{J}|^{2}\cdot\|\a\|_{L^2}^{2}.\\
		&\geq(\frac{1}{2}(1-c(n,k)c)^{2}c(n,k)^{-2}\|\theta\|^{-2}_{L^{\infty}(X)}-\sup|N_{J}|^{2})\|\a\|_{L^2}^{2}.
		\end{split}
		\end{equation*} 
		Rearrangement gives (\ref{E26}).
	\end{proof}

	\section{Applications}
	Having established a range of vanishing theorems and $L^{2}$-estimates, we now turn to applications. In particular, we demonstrate how these analytic tools can be used to study the topology of compact almost Kähler manifolds with negative sectional curvature, leading to new results on the Hirzebruch $\chi_{y}$-genus.
	\subsection{Hirzebruch $\chi_{y}$-genus }
We now turn to the definition of the Hirzebruch $\chi_{y}$-genus and its relation to the index theory of elliptic operators (cf.\cite{Huy}).	

Let $(X,J)$ be a closed almost complex manifold of real dimension $2n$ with an almost complex structure $J$.
The choice of an almost Hermitian metric on $X$ enables us to define the Hodge-star operator $\star$ and the formal adjoint $\bar{\pa}^{\ast}$ the $\bar{\pa}$-operator. For each $0\leq p\leq n$,  
we define an elliptic differential operator
\[\mathcal{D}_{p}:=\bar{\pa}+\bar{\pa}^{\ast}:\Om^{p,+}(X)\rightarrow\Om^{p,-}(X), \]
where 
\[\Om^{p,+}(X)=\bigoplus_{q=even}\Om^{p,q}(X)\quad and\quad \Om^{p,-}(X)=\bigoplus_{q=odd}\Om^{p,q}(X). \]
The index of this operator is denoted, following Hirzebruch's notation, by
\[ \chi_{p}(X):= \dim\ker(\bar{\pa}+\bar{\pa}^{\ast})-\dim{\rm{coker}}(\bar{\pa}+\bar{\pa}^{\ast}).\]
Applying the general Riemann--Roch--Hirzebruch theorem (first proved by Hirzebruch for projective manifolds and later extended to the general case by Atiyah and Singer),
we have the formula
\[ \chi_{p}(X)=\int_{X}\mathrm{ch}(\Om^{p,0}(X))\mathrm{Td}(TX),\]
where $\mathrm{ch}(\cdot)$ is the Chern character and $\mathrm{Td}(\cdot)$ is the Todd class. 
The Hirzebruch $\chi_{y}$-genus  $\chi_{y}(X)$ is then defined as the polynomial
\[ \chi_{y}(X)=\sum_{p=0}^{n}\chi_{p}(X)y^{p}.\]

Let $\gamma_{i}$, $1\leq i\leq n$, denote the formal Chern roots of the tangent bundle; that is, the $i$-th elementary symmetric polynomial of $\gamma_{1},\cdots,\gamma_{n}$ represents the $i$-th Chern class $c_{i}(X)$ of
$TX$, then the Atiyah--Singer Index Theorem yields the following expression:
\[\chi_{y}(X)=\int_{X}\prod_{i=1}^{n}(1+ye^{-\gamma_{i}})\frac{\gamma_{i}}{1-e^{-\gamma_{i}}}.\]
The Hirzebruch $\chi_{y}$-genus $\chi_{y}(X)$ interpolates between three fundamental topological invariants of $X$ (cf.\cite{Huy}): the Euler number (when $y=-1$), the Todd genus (when $y=0$), the signature (when $y=1$). 
\subsection{Atiyah's $L^{2}$-index theorem}

Let $\pi:(\tilde{X},\tilde{J})\rightarrow(X,J)$ be the universal covering with $\Gamma=\pi_{1}(X)$ as an isometry group of deck transformations. 
The Dolbeault-type operators $\mathcal{D}_{p}=\bar{\pa}+\bar{\pa}^{\ast}$ lift naturally to operators on $(\tilde{X},\tilde{J})$:
\[\tilde{\mathcal{D}}_{p}=\widetilde{\bar{\pa}+\bar{\pa}^{\ast}}:\Om^{p,+}_{(2)}(\tilde{X})\rightarrow\Om^{p,-}_{(2)}(\tilde{X}).\]

We denote by $\dim_{\Gamma}\ker(\tilde{\mathcal{D}}_{p})$ and $\dim_{\Gamma}\mathrm{coker}(\tilde{\mathcal{D}}_{p})$ 
the von Neumann dimensions of $\ker(\tilde{\mathcal{D}}_{p})$ and  $\mathrm{coker}(\tilde{\mathcal{D}}_{p})$ 
with respect to $\Gamma$, which are nonnegative real numbers \cite{Ati,CG,Luck}. 
The following fundamental property of the von Neumann dimension is essential (see \cite[Theorem 1.12]{Luck}):
\[ \dim_{\Gamma}\ker(\tilde{\mathcal{D}}_{p})=0\Leftrightarrow \ker(\tilde{\mathcal{D}}_{p})=0 ,\]
\[ \dim_{\Gamma}\mathrm{coker}(\tilde{\mathcal{D}}_{p})=0\Leftrightarrow \mathrm{coker}(\tilde{\mathcal{D}}_{p})=0.\]

The $L^{2}$-index of the lifted operators $\tilde{\mathcal{D}}_{p}$ is then defined as 
\begin{equation*}
L^{2}\mathrm{Index}_{\Gamma}(\tilde{\mathcal{D}}_{p})=\dim_{\Gamma}(\ker\tilde{\mathcal{D}}_{p})-\dim_{\Gamma}({\rm{coker}}\tilde{\mathcal{D}}_{p}).
\end{equation*}
We now recall the following Atiyah's $L^{2}$ index theorem \cite{Ati,Pan}.
\begin{theorem}\cite[Theorem 6.1]{Pan}
	Let $X$ be a closed Riemannian manifold and let $P$ be a determined elliptic operator on sections of certain bundles over $X$. 
	Denote  by $\tilde{\mathcal{D}}$ its lift of $\mathcal{D}$ to the universal convering space $\tilde{X}$, and let $\Ga=\pi_{1}(M)$. 
	Then the $L^{2}$ kernel of $\tilde{P}$ has a finite $\Ga$-dimension, and 
	$$L^{2}\mathrm{Index}_{\Ga}(\tilde{P})=\mathrm{Index}(P).$$
\end{theorem}

As an immediate consequence, we obtain the following relation  between the $\chi_{p}$-genus of $X$ and the $L^{2}$-index of the lifted operator.

\begin{proposition}\label{Pro2}
	Let $(X,J)$ be a closed almost complex manifold. Then the following equality holds:
	\[\chi_{p}(X) = L^{2}\mathrm{Index}_{\Gamma}(\tilde{\mathcal{D}}_{p}).\]
\end{proposition}

\begin{remark}
	When $J$ is integrable, i.e. $\bar{\pa}^{2}=0$, $\chi_{p}(X)$ equals the index of the following well-known Dolbeault complex
	$$\cdots\Om^{p,q-1}(X)\xrightarrow{\bar{\pa}}\Om^{p,q}(X)\xrightarrow{\bar{\pa}}\Om^{p,q+1}(X)\rightarrow\cdots$$
	and consequently
	$$\chi_{p}(X)=\sum_{q=0}^{n}(-1)^{q}h^{p,q}(X),$$
	where $h^{p,q}(X)$ denote the Hodge numbers of  $X$. 
	
	In the integrable case, $\chi_{p}(X)$ reduces to the alternating sum of Hodge numbers, which are known to satisfy various inequalities under non-positive (or negative) curvature assumptions (see \cite{CF,Gro}). Our goal is to extend such results to the almost Kähler setting, where no Dolbeault cohomology exists, but the index $\chi_{p}(X)$ remains well-defined.	
\end{remark}

	\subsection{The case of small Nijenhuis tensor}
	The following theorem recovers the estimates of Theorem \ref{T11} in the special case $c=0$, i.e. when $\w$ itself is exact and bounded. This situation occurs, for example, on hyperbolic manifolds after lifting to the universal cover, as we will see shortly.
	\begin{theorem}\label{Thm6}
	Let $(X,J,\w)$ be a complete $2n$-dimensional almost K\"{a}hler manifold. Suppose that there exists a bounded $1$-form $\theta$ such that
	\[\w=d\theta.\]
	Then for any $\a\in\Om^{p,q}_{0}(X)$ with $k:=p+q\neq n$, the following estimates hold:
	\begin{equation*}
	\begin{split}
	&\|\a\|_{L^{2}(X)}\leq c_2(n,k)\|\theta\|_{L^{\infty}(X)}\Big((\De_{\pa}+\De_{\bar{\pa}})\a,\a\Big)^{\frac{1}{2}},\\
	&\|\a\|_{L^{2}(X)}^{2}\Big(1-c_2(n,k)^{2}\|\theta\|^{2}_{L^{\infty}(X)}\sup|N_{J}|^{2}\Big)
	\leq 2c_2(n,k)^{2}\|\theta\|^{2}_{L^{\infty}(X)}(\De_{\bullet}\a,\a),
	\end{split}
	\end{equation*}
	where $\bullet=\pa,\bar{\pa}$. 
\end{theorem}	
	\begin{proof}
	This conclusion can be deduced immediately from Theorem \ref{T11} (with $c=0$ in this case)
	\end{proof}
	Next, we establish an $L^{2}$-estimate for the operator $\bar{\partial} + \bar{\partial}^{*}$ under the same hypotheses. The proof of Theorem \ref{Thm3} is technically more involved because it deals directly with the operator $\bar{\partial}+\bar{\partial}^{\ast}$ whose square is not simply the sum of Laplacians. The key is to exploit the identity $\bar{\pa}^{2}=-[\pa,\bar{\mu}]$ and use the almost Kähler identities to control the resulting cross terms.
\begin{theorem}\label{Thm3}
	Let $(X,J,\w)$ be a complete $2n$-dimensional almost K\"{a}hler manifold. Suppose that there exists a bounded $1$-form $\theta$ such that
	\[ \w=d\theta.\]
	Then there exist constants $c_{3}(n),c_{4}(n)$ depending only on $n$ such that the following holds: 
	
	for any $\a\in\Om_{0}^{p,+}(X)$ when $n-p$ is odd (resp. $\a\in\Om_{0}^{p,-}(X)$ when $n-p$ is even), 
	\[\|\theta\|^{2}_{L^{\infty}(X)}\|(\bar{\partial}+\bar{\partial}^{\ast})\a\|_{L^{2}(X)}
	\geq \Big(c_{3}(n)-c_{4}(n)\|\theta\|^{2}_{L^{\infty}(X)}\sup|N_{J}|^{2}\Big)\|\a\|^{2}_{L^{2}(X)}. \]
\end{theorem}

\begin{proof}
	We give the proof for the case when $n-p$ is odd; the even case is entirely analogous. 
	
	First, observe that
	\[({\bar{\partial}+\bar{\partial}^{\ast}})^{2}=\De_{\bar{\pa}}+(\bar{\pa})^{2}+(\bar{\pa}^{\ast})^{2}. \]
	For any $\a\in\Om_{0}^{p,+}(X)$, we decompose  $\a:=\sum\a_{2i}$ with $\a_{2i}\in\Om^{p,2i}$. Then
	\begin{equation*}
	\begin{split}
	\|({\bar{\partial}+\bar{\partial}^{\ast}})\a\|_{L^{2}(X)}^{2}
	&=(\De_{\bar{\pa}}\a,\a)+\Big((\bar{\pa})^{2}\a,\a\Big)+\Big((\bar{\pa}^{\ast})^2\a,\a\Big)\\
	&=(\De_{\bar{\pa}}\a,\a)+2\mathrm{Re}\Big((\bar{\pa})^{2}\a,\a\Big)\\
	&=\sum(\De_{\bar{\pa}}\a_{2i},\a_{2i})+2\sum\mathrm{Re}\Big((\bar{\pa})^{2}\a_{2i},\a_{2i+2}\Big)\\
	&=\sum(\De_{\bar{\pa}}\a_{2i},\a_{2i})-2\sum\mathrm{Re}\Big([\pa,\bar{\mu}]\a_{2i},\a_{2i+2}\Big),
	\end{split}
	\end{equation*}
	where we used the identity (cf.(\ref{E27}))
	\[ \bar{\pa}^{2}=-[\pa,\bar{\mu}].\]
	
	For each term in the sum, we estimate:
	\begin{equation*}
	\begin{split}
	|([\pa,\bar{\mu}]\a_{2i},\a_{2i+2})|&=|(\bar{\mu}\a_{2i},\pa^{\ast}\a_{2i+2})+(\pa\a_{2i},\bar{\mu}^{\ast}\a_{2i+2})|\\
	&\leq \|\bar{\mu}\a_{2i}\|_{L^{2}}\|\pa^{\ast}\a_{2i+2}\|_{L^{2}}+ \|\pa\a_{2i}\|_{L^{2}}\|\bar{\mu}^{\ast}\a_{2i+2}\|_{L^{2}}.\\
	\end{split}
	\end{equation*}
	Summing over $i$ and applying the elementary inequality $2ab\leq c_{0}a^{2}+\frac{1}{c_{0}}b^{2}$
	for all $c_{0}>0$, we obtain
	\begin{equation*}
	\begin{split}
	|2\sum\mathrm{Re}([\pa,\bar{\mu}]\a_{2i},\a_{2i+2})|
	&\leq2\sum(\|\bar{\mu}\a_{2i}\|_{L^{2}}\|\pa^{\ast}\a_{2i+2}\|_{L^{2}}+ \|\pa\a_{2i}\|_{L^{2}}\|\bar{\mu}^{\ast}\a_{2i+2}\|_{L^{2}})\\
	&\leq \sum c_{0}(\|\pa\a_{2i}\|^2_{L^{2}}+\|\pa^{\ast}\a_{2i}\|^2_{L^{2}})
	+\sum\frac{1}{c_{0}}(\|\bar{\mu}\a_{2i}\|^2_{L^{2}}+\|\bar{\mu}^{\ast}\a_{2i}\|^2_{L^{2}})\\
	&=\sum c_{0}(\De_{\pa}\a_{2i},\a_{2i})+\sum\frac{1}{c_{0}}(\De_{\bar{\mu}}\a_{2i},\a_{2i}).\\
	\end{split}
	\end{equation*}	
	Combining this with the earlier expression for $\|({\bar{\partial}+\bar{\partial}^{\ast}})\a\|^{2}$ and using the identity $$\De_{\bar{\pa}}=\De_{\pa}+\De_{\bar{\mu}}-\De_{\mu},$$
	we find 
	\begin{equation*}
	\begin{split}
	\|({\bar{\partial}+\bar{\partial}^{\ast}})\a\|_{L^{2}(X)}^{2}&\geq\sum\big{(}(\De_{\bar{\pa}}\a_{2i},\a_{2i})- c_{0}(\De_{\pa}\a_{2i},\a_{2i})-\frac{1}{c_{0}}(\De_{\bar{\mu}}\a_{2i},\a_{2i})\big{)}\\
	&=\sum\big{(}(1-c_{0})(\De_{\bar{\pa}}\a_{2i},\a_{2i})- c_{0}(\De_{\mu}\a_{2i},\a_{2i})-(\frac{1}{c_{0}}-c_{0})(\De_{\bar{\mu}}\a_{2i},\a_{2i})\big{)}.\\
	\end{split}
	\end{equation*}
	
	Choosing $c_{0} = \frac{1}{2}$ and applying Theorem \ref{Thm6}, we obtain
	\begin{equation*}
	\begin{split}
	\|({\bar{\partial}+\bar{\partial}^{\ast}})\a\|_{L^{2}(X)}^{2}
	&\geq\sum\Big(\frac{1}{2}(\De_{\bar{\pa}}\a_{2i},\a_{2i})-\frac{1}{2}(\De_{\mu}\a_{2i},\a_{2i})-\frac{3}{2}(\De_{\bar{\mu}}\a_{2i},\a_{2i})\Big)\\
	&\geq \sum \|\theta\|^{-2}_{L^{\infty}(X)}\Big(c_{3}(n)-c_{4}(n)\|\theta\|^{2}_{L^{\infty}(X)}\sup|N_J|^{2}\Big)\|\a_{2i}\|^{2}_{L^{2}(X)}\\
	&= \|\theta\|^{-2}_{L^{\infty}(X)}\Big(c_{3}(n)-c_{4}(n)\|\theta\|^{2}_{L^{\infty}(X)}\sup|N_J|^{2}\Big)\|\a\|_{L^{2}(X)}^{2}.
	\end{split}
	\end{equation*}
	This completes the proof.
\end{proof}

Let $\mathcal{E}_{1}$ and $\mathcal{E}_{2}$ be $C^{\infty}$-vector bundles over a smooth manifold $X$, 
and let $\mathcal{D}:C^{\infty}(\mathcal{E}_{1})\rightarrow C^{\infty}(\mathcal{E}_{2})$ be a differential operator between $C^{\infty}$-sections of these bundle. 
Suppose furthermore that $X$ is a smooth Riemannian manifold and that $\Ga$ is a discrete group of isometrics of $X$ such that 
the differential operator $\mathcal{D}$ commutes with the action of $\Ga$. Let $(\mathcal{L},\na)$ be a $\Ga$-invariant Hermitian line bundle on $X$, 
and assume that $X/\Ga$ is closed. Under these hypotheses, we recall Atiyah's $L^{2}$-index theorem for the twisted operator $\mathcal{D}\otimes\na$ (cf. \cite{Gro}).

\begin{theorem}\cite[Theorem 2.3.A]{Gro}\label{Thm5}
	Let $\mathcal{D}$ be a first-order elliptic operator. Then there exists a closed nonhomogeneous form
	\[I_{D}=I^{0}+I^{1}+\cdots+I^{n}\in\Om^{\ast}(X)=\Om^{0}\oplus\Om^{1}\oplus\cdots\oplus\Om^{n} \]
	invariant under $\Ga$, such that the $L^{2}$-index of the twisted operator $\mathcal{D}\otimes\na$ is given by
	\[L^{2}\mathrm{Index}_{\Ga}(\mathcal{D}\otimes\na)=\int_{X/\Ga}I_{\mathcal{D}}\wedge\exp{[\w]}, \]
	where $[\w]$ is the Chern form of $\na$, and
	\[\exp{[\w]}=1+[\w]+\frac{[\w]\wedge[\w]}{2!}+\frac{[\w]\wedge[\w]\wedge[\w]}{3!}+\cdots. \]
\end{theorem}

\begin{remark}\label{Rem}
	
	(1)	An non-zero $L^{2}\mathrm{Index}_{\Ga}(\mathcal{D}\otimes\na)$ implies that 
	either $\mathcal{D}\otimes\na$ or its adjoint admits a non-trivial $L^{2}$-kernel.
	
	(2) In the present article, the operator $\mathcal{D}$ is taken to be $\bar{\pa}+\bar{\pa}^{\ast}$. 
	In this case, the $I^{0}$-component of $I_{\mathcal{D}}$ is nonzero. 
	Consequently, the integral $\int_{X/\Ga}I_{\mathcal{D}}\wedge\exp{t[\w]}\neq 0$, for almost all $t$, 
	provided the curvature form $\w$ is homologically nonsingular $\int_{X/\Ga}\w^{n}\neq 0$, for $n=\dim_{\C}X$ (cf. \cite[2.3.$\mathrm{A}'$. Remarks]{Gro}).
\end{remark}

Gromov defined the lower spectral bound $\la_{0}=\la_{0}(\mathcal{D})\geq 0$ as the upper bound of the negative numbers $\la$, such that 
\[\|\mathcal{D}s\|_{L^{2}}\geq\la\|s\|_{L^{2}} \]
for those sections $e$ of $\mathcal{E}$ where $\mathcal{D}s$ in $L^{2}$. 
Now let $\mathcal{D}$ be a $\Ga$-invariant, first-order elliptic operator on $X$, 
and let $I_{D}=I^{0}+I^{1}+\cdots+I^{n}\in\Om^{\ast}(X)$ be the corresponding index form. 
Let $\w$ be a closed $\Ga$-invariant $2$-form on $X$ and denote by $I_{t}^{n}$ the top component of product $I_{\mathcal{D}}\wedge\exp{t\w}$, 
for $t\in\mathbb{R}$. Hence $I_{t}^{n}$ is a $\Ga$-invariant $n$-form on $X$, $\dim X=n$ depending on parameter $t$.  The following theorem provides a criterion for the vanishing of the lower spectral bound.
\begin{theorem}(\cite[2.4.A. Theorem]{Gro})\label{Thm4}
	Assume that $H^{1}_{dR}(X)=0$, that $X/\Ga$ is closed, and that $\int_{X/\Ga}I_{t}^{n}\neq 0$, for some $t\in\mathbb{R}$. 
	If the form $\w$ is $d$(bounded), then either $\la_{0}(\mathcal{D})=0$ or $\la_{0}(\mathcal{D}^{\ast})=0$,
	where $\mathcal{D}^{\ast}$ is the adjoint operator.
\end{theorem}
We also recall the following classical result due to Gromov, for which Chen and Yang provided a detailed proof.
\begin{lemma}(\cite[Lemma 3.2]{CY})\label{L7}
	Let $(X,g)$ be a complete simply-connected manifold with negative sectional curvature bounded from above by a negative constant, i.e.
	$$\mathrm{sec}\leq -K$$
	for some $K>0$. Then for any bounded and closed $k$-form $\a$ on $X$, where $k\geq2$, there exists a bounded $(k-1)$-form $\b$ such that 
	$$\a=d\b$$
	and
	\[ \|\b\|_{L^{\infty}}\leq K^{-\frac{1}{2}}\|\a\|_{L^{\infty}}.\]
\end{lemma}
We now apply the above general theory to the specific setting of almost Kähler manifolds with negative sectional curvature. Corollary \ref{C5} applies Theorem \ref{Thm3} to the universal cover of a negatively curved almost Kähler manifold. 
\begin{corollary}\label{C5}
	Let $(X,J,\w)$ be a closed $2n$-dimensional almost K\"{a}hler manifold with  negative sectional curvature,
	i.e. there exists a constant $K>0$ such that \[ \mathrm{sec}\leq -K.\]
	Let $\pi:(\tilde{X},\tilde{J},\tilde{\w})\rightarrow(X,J,\w)$ be the universal covering map of $X$.  
	Then for any $\a\in\Om_{0}^{p,+}(\tilde{X})$ when $n-p$ is odd (resp. $\a\in\Om_{0}^{p,-}(\tilde{X})$ when $n-p$ is even),
	the following estimate holds:
	\[\|(\widetilde{\bar{\partial}+\bar{\partial}^{\ast}})\a\|_{L^{2}(\tilde{X})}\geq \Big(\frac{c_{3}(n)}{n}K-c_{4}(n)\sup|N_{J}|^{2}\Big)\|\a\|^{2}_{L^{2}(\tilde{X})}. \]
\end{corollary}

\begin{proof}
	Note that the pointwise norm of the Kähler form satisfies $|\tilde{\w}|^{2}=n$. By Lemma \ref{L7}, the negative curvature condition implies the existence of a $1$-form $\theta$ such that 
	\[ \tilde{\w}=d\theta\quad and\quad \|\theta\|_{L^{\infty}(\tilde{X})}\leq K^{-\frac{1}{2}}\sqrt{n}.\]
	Applying Theorem \ref{Thm3} on the universal cover $\tilde{X}$ with this $1$-form $\theta$, we obtain
	\begin{equation*}
	\begin{split}
	nK^{-1}\|(\widetilde{\bar{\partial}+\bar{\partial}^{\ast}})\a\|_{L^{2}(\tilde{X})}
	&\geq\|\theta\|^{2}_{L^{\infty}(\tilde{X})}\|(\widetilde{\bar{\partial}+\bar{\partial}^{\ast}})\a\|_{L^{2}(\tilde{X})}\\
	&\geq  \Big(c_{3}(n)-c_{4}(n)\|\theta\|^{2}_{L^{\infty}(\tilde{X})}\sup|N_{J}|^{2}\Big)\|\a\|^{2}_{L^{2}(\tilde{X})}\\
	&\geq \Big(c_{3}(n)-c_{4}(n)K^{-1}n\sup|N_{J}|^{2}\Big)\|\a\|^{2}_{L^{2}(\tilde{X})}.
	\end{split}
	\end{equation*}
	This completes the proof.
\end{proof}

We are now in a position to prove the main theorem of this section.
\begin{proof}[\textbf{Proof of Theorem \ref{Thm2}}]
	Assume that the Nijenhuis tensor satisfies the smallness condition
	\[ \sup|N_{J}|^{2}\leq\frac{c_{3}(n)K}{2nc_{4}(n)}:=C(n)K.\]
	Substituting this bound into the estimate from Corollary \ref{C5}, we find that when $n-p$ is odd, for any $\a\in\Om^{p,+}_{0}(\tilde{X})$,  
	\[\|(\widetilde{\bar{\partial}+\bar{\partial}^{\ast}})\a\|^{2}_{L^{2}(\tilde{X})}
	\geq \Big(\frac{c_{3}(n)K}{2n}\Big)\|\a\|^{2}_{L^{2}(\tilde{X})}. \]
	This strictly positive lower bound implies that the lower spectral bound 
	$\la_{0}$ of the operator $ \widetilde{\bar{\partial}+\bar{\partial}^{\ast}}:\Om_{(2)}^{p,+}(\tilde{X})\rightarrow\Om_{(2)}^{p,-}(\tilde{X})$ 
	is strictly positive. 
	
	On the other hand, since $H^{1}_{dR}(\tilde{X})=0$ for a simply-connected manifold, $X=\tilde{X}/\Ga$ is closed, 
	the negative sectional curvature implies that the Kähler form \(\tilde{\omega}\) is $d$(bounded), 
	and the integral $\int_{X}I^{n}_{t}\neq 0$ for some $t\in\mathbb{R}$, (see Remark \ref{Rem}), we can apply Theorem \ref{Thm4}. The positivity of $\la_{0}$ forces the kernel of the operator itself to be trivial, which, by the Theorem \ref{Thm4}, forces the kernel of the adjoint operator to be non-trivial; specifically,
	\[{\rm{coker}}(\widetilde{\bar{\partial}+\bar{\partial}^{\ast}})\cap\Om^{p,odd}_{(2)}(\tilde{X})\neq \{0\}. \]
	Applying Proposition \ref{Pro2}, we conclude that
	\begin{equation*}	
	\chi_{p}(X)=-\dim_{\Gamma}\mathrm{coker}(\tilde{\mathcal{D}}_{p})<0.
	\end{equation*}
	
	When $n-p$ is even, an analogous argument (using the spaces $\Om^{p,-}_{(2)}(\tilde{X})$) shows that $\chi_{p}(M)>0$. 
	Combining both cases, we obtain the uniform inequality
	\[ (-1)^{n-p}\chi_{p}(X)\geq1\]
	for all $p=0,1,\cdots,n$. The statement for the Euler number follows by evaluating the Hirzebruch genus at $y=-1$, where
	\[\chi_{y}(X)|_{y=-1}=(-1)^{n}\chi(X)=\sum_{p=0}^{n}(-1)^{p}\chi_{p}(X).\]
	Since each term $(-1)^{n-p}\chi_{p}(X)\geq1$, their sum is at least $n+1$, giving $(-1)^{n}\chi(X)\geq n+1$.
\end{proof}

	\section*{Acknowledgements}
	This work is supported by the National Natural Science Foundation of China Nos. 12271496, 11801539 (Huang), 11701226 (Tan), 12201001 (Zhang) and the Youth Innovation Promotion Association CAS, the Fundamental Research Funds of the Central Universities, the USTC Research Funds of the Double First-Class Initiative. The authors also thank DeepSeek for its assistance in proofreading and improving the grammar and expression of this manuscript.

\end{document}